\documentclass[12pt,a4paper,twoside]{amsart}
\usepackage[latin1]{inputenc}
\usepackage[T1]{fontenc}
\usepackage{eucal}
\usepackage[english]{babel}
\usepackage{amssymb}
\usepackage{amsmath}
\usepackage{amsthm}
\usepackage{xr}
\usepackage{lmodern}
\usepackage{xcolor}

\title{Rational points close to non-singular algebraic curves}
\author{Faustin Adiceam}
\address{%
Laboratoire d'analyse et de math\'{e}matiques appliqu\'{e}es \\ Universit\'{e} Paris-Est Cr\'{e}teil \\ Cr\'{e}teil, France}
\email{%
faustin.adiceam@u-pec.fr}

\author{Oscar Marmon}
\address{Centre for Mathematical Sciences \\ Lund University \\ Box 118 \\ 221 00 Lund \\ Sweden}
\email{oscar.marmon@math.lu.se}

\begin{document}

\newcommand{\epsi}{\varepsilon}
\newcommand{\xx}{\mathbf{x}}
\newcommand{\xxi}{\boldsymbol{\xi}}
\newcommand{\eeta}{\boldsymbol{\eta}}
\newcommand{\0}{\boldsymbol{0}}
\newcommand{\bb}{\mathbf{b}}
\newcommand{\uu}{\mathbf{u}}
\newcommand{\ee}{\mathbf{e}}
\newcommand{\vv}{\mathbf{v}}
\newcommand{\yy}{\mathbf{y}}
\newcommand{\zz}{\mathbf{z}}
\newcommand{\ww}{\mathbf{w}}
\newcommand{\ff}{\mathbf{f}}
\newcommand{\cc}{\mathbf{c}}
\newcommand{\dd}{\mathbf{d}}
\newcommand{\hh}{\mathbf{h}}
\newcommand{\ttt}{\mathbf{t}}
\renewcommand{\aa}{\mathbf{a}}
\newcommand{\bK}{\mathbf{K}}
\newcommand{\bB}{\mathbf{B}}
\newcommand{\ZZ}{\mathbb{Z}}
\newcommand{\Zpol}{\ZZ[X_1,\ldots,X_n]}
\newcommand{\FF}{\mathbb{F}}
\newcommand{\RR}{\mathbb{R}}
\newcommand{\NN}{\mathbb{N}}
\newcommand{\CC}{\mathbb{C}}
\renewcommand{\AA}{\mathbb{A}}
\newcommand{\PP}{\mathbb{P}}
\newcommand{\GG}{\mathbb{G}}
\newcommand{\QQ}{\mathbb{Q}}
\newcommand{\sB}{\mathsf{B}}
\newcommand{\cP}{\mathcal{P}}
\newcommand{\cR}{\mathcal{R}}
\newcommand{\cB}{\mathcal{B}}
\newcommand{\cE}{\mathcal{E}}
\newcommand{\cC}{\mathcal{C}}
\newcommand{\cN}{\mathcal{N}}
\newcommand{\cM}{\mathcal{M}}
\newcommand{\cL}{\mathcal{L}}
\newcommand{\cD}{\mathcal{D}}
\newcommand{\cA}{\mathcal{A}}
\newcommand{\cK}{\mathcal{K}}
\newcommand{\cF}{\mathcal{F}}
\newcommand{\cZ}{\mathcal{Z}}
\newcommand{\cX}{\mathcal{X}}
\newcommand{\cG}{\mathcal{G}}
\newcommand{\cH}{\mathcal{H}}
\newcommand{\cI}{\mathcal{I}}
\newcommand{\cO}{\mathcal{O}}
\newcommand{\cW}{\mathcal{W}}
\newcommand{\ud}{\mathrm{ud}}
\newcommand{\Zar}{\mathrm{Zar}}
\newcommand{\fm}{\mathfrak{m}}
\newcommand{\fp}{\mathfrak{p}}
\newcommand{\smod}[1]{\,(#1)}
\newcommand{\spmod}[1]{\,(\mathrm{mod}\,{#1})}

\renewcommand{\bf}{\mathbf}
\newcommand{\cal}{\mathcal}
\newcommand{\prim}{\mathrm{prim}}
\newcommand{\GL}{\mathrm{GL}}
\newcommand{\tr}{\mathrm{tr}}
\newcommand{\m}{\mathfrak{m}}
\newcommand{\cS}{\mathcal{S}}
\newcommand{\Mod}[1]{\ (\mathrm{mod}\ #1)}
\renewcommand{\phi}{\varphi}

\newtheorem{thm}{Theorem}[section]
\newtheorem{lemma}[thm]{Lemma}
\newtheorem*{lemma*}{Lemma}
\newtheorem{prop}[thm]{Proposition}
\newtheorem*{prop*}{Proposition}
\newtheorem*{thm*}{Theorem}
\newtheorem{claim}[thm]{Claim}
\newtheorem{cor}[thm]{Corollary}
\newtheorem*{conj*}{Conjecture}
\newtheorem{conj}{Conjecture}
\newtheorem*{sats*}{Sats}
\theoremstyle{remark}
\newtheorem*{note*}{Note}
\newtheorem{note}{Note}
\newtheorem*{rem*}{Remark}
\newtheorem{rem}[thm]{Remark}
\newtheorem{example}[thm]{Example}
\newtheorem*{acknowledgement*}{Acknowledgement}
\newtheorem*{question*}{Question}
\newtheorem*{answer*}{Answer}
\theoremstyle{definition}
\newtheorem*{def*}{Definition}
\newtheorem{notation}[thm]{Notation}
\newtheorem*{notation*}{Notation}
\newtheorem{assumptions}[thm]{Assumptions}
\newtheorem{assumption}[thm]{Assumption}

\begin{abstract}
We study the density of solutions to Diophantine inequalities involving non-singular ternary forms, or equivalently, the density of rational points close to non-singular plane algebraic curves.
\end{abstract}

\maketitle

\section{Introduction}
\label{sec:intro}

Let $F \in \ZZ[x,y,z]$ be a homogeneous polynomial of degree $k \geq 3$. We wish to estimate the number $N_\gamma(F,B)$ of primitive integer solutions to the inequalities
\begin{equation}
\label{eq:ternary_ineq}
|F(x,y,z)| \leq B^\gamma, \quad B/2 < \max\{|x|,|y|,|z|\} \leq B,
\end{equation}
where 
$0 \leq \gamma < k$. 
We shall always assume that $F$ is non-singular, that is, that $\nabla F(\xx) \neq \0$ for all non-zero $\xx \in {\overline\QQ}^3$. It is natural to view solutions to \eqref{eq:ternary_ineq} as rational points close to the projective curve defined by the equation $F=0$, which we shall denote by $C_F \subset \PP^2$. It is clearly enough to count solutions to \eqref{eq:ternary_ineq} for which 
\begin{equation}
\label{eq:z_biggest}
|z| \geq \max\{|x|,|y|\}.
\end{equation}
Then, putting
\[
g(x,y) = F(x,y,1),
\]
we are equivalently counting rational points 
\[
(r,s) = (x/z,y/z) \in [-1,1]^2,
\]
with denominator $z \in (B/2,B]$, that satisfy 
$|g(r,s)| =O(B^{\gamma-k})$. These points in turn lie at a distance of at most $O(B^{\gamma-k})$ from the affine curve $g(x,y) = 0$, in a sense that will be made precise in \S \ref{sec:proof}.

An \emph{inflection point} of a projective plane curve is a point where the curve meets its tangent line with multiplicity at least three. A \emph{sextactic point} is a point where the curve meets its osculating conic with multiplicity at least six. Such points always exist, but a curve of degree $k \geq 3$ can have at most $3k(k-2)$ inflection points and $3k(4k-9)$ sextactic points, see for example \cite{Thorbergsson_Umehara}.
The tangent lines at inflection points may contain an excessive amount of rational points close to the curve, as seen by considering an example such as the Fermat curve $x^k + y^k = 1$ with its tangent line $y=1$ at $(0,1)$. 

With this in mind, we enumerate the tangent lines to $C_F$ at inflection points as $T_1,\dotsc,T_K$, where $K \leq 3k(k-2)$, and similarly the osculating conics to $C_F$ at its sextactic points as $O_1,\dotsc,O_{K'}$, where $K' \leq 3k(4k-9)$. We define the modified counting function $N^*_\gamma(F,B)$ to be the number of solutions to \eqref{eq:ternary_ineq} such that
\[
(x,y,z) \notin \bigcup_{1 \leq i \leq K} T_i. 
\]
The osculating conics $O_i$ will require special care, but we shall see that their contribution is small enough that they do not need to be excluded from the point count. Furthermore, we define the quantity $\nu = \nu_F$ to be the maximum order of contact of $C_F$ with any of the tangent lines $T_i$. We then have 
\(
3 \leq \nu_F \leq k,
\)
and it is well known that 
\begin{equation}
\label{inflection_assumptions}
\nu_F = 3
\end{equation}
for a generic choice of $F$, in the following sense: a general plane curve of degree $k$ has no hyperflex -- if $F$ is viewed as an element of the parameter space $\PP^{N-1}$, where $N = \binom{k+2}{2}$, then the locus of curves $C_F$ violating \eqref{inflection_assumptions} is a proper Zariski closed subset of $\PP^{N-1}$. We shall sometimes make the assumption \eqref{inflection_assumptions} to obtain the best possible bounds.

Before stating our bounds for the counting functions $N_\gamma(F,B)$ and $N^*_\gamma(F,B)$, it is perhaps good to recall the best available bound for the number of rational points of bounded height \emph{on} the curve $C_F$. If
\[
N_F(B) = \#\{(x,y,z) \in \ZZ^3_\prim: F(x,y,z) = 0,\ |x|,|y|,|z| \leq B\},
\]
then Heath-Brown \cite{Heath-Brown02} has shown that
\begin{equation}
\label{eq:on_curve}
N_F(B) \ll_{k,\epsi} B^{2/k + \epsi}
\end{equation}
for any $\epsi >0$. This bound has subsequently been improved by Walsh \cite{Walsh} to $N_F(B) \ll_k B^{2/k}$, which is sharp. Here we have used the Vinogradov notation $f \ll g$ to express the asymptotic relation $f=O(g)$, with subscripts $k,\epsi$ to indicate that the implied constant is allowed to depend on those parameters (but not on the the coefficients of $F$). Related notations that will also be used in the paper are: $f \gg g$ to denote that $g \ll f$, $f \asymp g$ to denote that $g \ll f \ll g$, $f \sim g$ to denote that $f/g \to 1$ (as $B \to \infty$), and $f \approx g$ to denote that $f \sim C g$ for some constant $C$.

Our first result is a general bound for the modified counting function $N^*_\gamma(F,B)$. It is most conveniently stated by first introducing the quantity
\begin{equation}
\label{eq:tau_def}
\tau = k-\gamma. 
\end{equation}

\begin{thm}
\label{thm:main}
Suppose that $k \geq 5$ and $\tau \geq 2$. Then we have
\begin{multline*}
N^*_\gamma(F,B) \ll_{F,\gamma,\epsi} B^{9/(4\tau) + 1-\tau/k + \epsi} + B^{2-\tau/4+ \epsi} + B^{2 + 27/(20\tau) - 9\tau/20 + \epsi}
\\+ B^{3/(2\tau) + 2/3 + \epsi}.
\end{multline*}
In particular we have
\[
N^*_\gamma(F,B) \ll_{F,\gamma,\epsi} B^{9/(4\tau)+ 1-\tau/k + \epsi} + B^{3/(2\tau) + 2/3 + \epsi}
\]
for $\tau \geq 3.721$.
\end{thm}

\begin{rem*}
Just as in \eqref{eq:on_curve}, our estimates get sharper as the degree $k$ increases. As our bounds are quite weak for low degrees, we have not stated the result that one would get for $k=3$ or $4$.
\end{rem*}

One may view our results in light of the recent progress on rational points close to smooth curves. For a twice continuously differentiable function $f$ defined on a compact interval $I \subset \RR$, define $\tilde{\cN}_f(B,\delta)$ to be the number of triples $(a,b,q) \in \ZZ^3_\prim$ satisfying the conditions
\begin{equation}
\label{eq:smoothcurve_countingproblem}
B/2 \leq q \leq B,\quad \frac{a}{q} \in I,\quad \left|\frac{b}{q}-f\left(\frac{a}{q}\right)\right| \leq \frac{\delta}{B}.
\end{equation}
Given a bound $1 \ll |f''| \ll 1$ on $I$, Huxley \cite{Huxley94} proved that
\begin{equation}
\label{eq:huxley}
\tilde{\cN}_f(B,\delta) \ll \delta^{1-\epsi} B^2 + B.
\end{equation}
In particular, this implies that
\begin{equation}
\label{eq:huxley_smalldelta}
\tilde{\cN}_f(B,\delta) \ll B^{1+\epsi}
\end{equation}
whenever $\delta \leq B^{-1}$.
Refinements by Beresnevich, Dickinson and Velani \cite{Beresnevich-Dickinson-Velani}, Huang \cite{Huang15} and Chow \cite{Chow16} in fact give an asymptotic formula 
\(
\tilde{\cN}_f(B,\delta) \approx \delta B^2
\)
as soon as $\delta \geq B^{-1+\epsi}$. On the other hand, as discussed in \cite{Huang15},  for smaller values of $\delta$, analytic geometry is not enough to characterise the behaviour of $\tilde{\cN}_f(B,\delta)$. As we shall see in \S \ref{sec:proof}, if we restrict to suitable affine patches of the projective curve $C_F$ that are expressed as graphs of smooth functions $f$, then the points counted by $N_\gamma(F,B)$ or $N^*_\gamma(F,B)$ correspond to solutions to the counting problem \eqref{eq:smoothcurve_countingproblem} in which $\delta \approx B^{1-\tau}$. Consequently, we have restricted our attention to the less understood range $\tau \geq 2$ in all of our results. 

In view of the bound \eqref{eq:huxley_smalldelta}, we formulate the following consequence of Theorem \ref{thm:main}. 
\begin{cor}
\label{cor:main}
We have 
\[
N^*_\gamma(F,B) = o(B)
\]
if either $5 \leq k \leq 9$ and $\tau > \frac{9}{2}$ or else $k>9$ and $\tau > \frac{3}{2} \sqrt{k}$.
\end{cor}

One interpretation of the bound in Corollary \ref{cor:main} is that, for algebraic curves, it is possible to go beyond the exponent 1 in \eqref{eq:huxley_smalldelta}. However, another important aspect is that the cited bounds for smooth curves are local in nature, whereas our Theorem \ref{thm:main} is a global result. Indeed, the curvature bound $1 \ll |f''| \ll 1$ can never hold globally on the curve $C_F$ due to the existence of inflection points. 

Let us also discuss the contribution from the exceptional lines 
excluded from consideration in Theorem \ref{thm:main}. 
It is easy to see that the finitely many tangent lines $T_j$ at inflection points contribute at most
\begin{equation}
\label{eq:tangent}
\ll B^{2-\tau/\nu}
\end{equation}
points to $N_\gamma(F,B)$, with $\nu = \nu_F \in [3,k]$ as above. This bound is sharp.
If we impose the condition \eqref{inflection_assumptions}, then the contribution \eqref{eq:tangent} is negligible compared to the term $B^{3/(2\tau)+2/3 + \epsi}$. Some of the other terms in our bound can also be improved under that assumption, giving the following results.
\begin{thm}
\label{thm:main_generic}
Suppose that $k \geq 5$ and $\tau \geq 2$, and that the condition \eqref{inflection_assumptions} holds. Then we have
\[
N_\gamma(F,B) \ll B^{9/(4\tau)+ 1-\tau/k + \epsi} + B^{3/(2\tau) + 2/3 + \epsi}.
\]
\end{thm}
\begin{cor}
\label{cor:generic}
Under the assumption \eqref{inflection_assumptions}, we have 
\[
N_\gamma(F,B) = o(B)
\]
if either $5 \leq k \leq 9$ and $\tau > \frac{9}{2}$ or else $k>9$ and $\tau > \frac{3}{2} \sqrt{k}$.
\end{cor}

Here we have suppressed the dependence on $F,\gamma$ and $\epsi$ in the implied constants, as we shall mostly do in the rest of the paper. Instead we shall occasionally use a notation such as $\ll_k$ to indicate a bound where the implied constant is independent of the coefficients of $F$.

\begin{rem*}
One may ask how our bounds compare to the more naive approach of separately counting rational points on each of the level surfaces $F(x,y,z) = c$, for $|c| \leq B^\gamma$. As explained in the introduction to \cite{Heath-Brown09}, the contribution from each such surface can be shown to be of order $B^{2/\sqrt{k}+\epsi}$. This would give a bound for $N_\gamma(F,B)$ of order $B^{\gamma + 2/\sqrt{k}+\epsi}$, which is clearly inferior to the bounds presented here in all cases. 
\end{rem*}

Our main tool in the proof of Theorem \ref{thm:main} will be the approximate determinant method introduced by Heath-Brown in \cite{Heath-Brown09}. We develop this method in a general setting in \S \ref{sec:determinant}. In the course of the investigation, we will also need to consider the corresponding problem for binary forms, a Thue inequality. In \S \ref{sec:binary}, we derive some bounds for the number of solutions to such inequalities that will be needed. The pieces are put together in \S 4 to give the final estimate. The outcome of the determinant method is that the points counted by $N_\gamma(F,B)$ lie on a number of auxiliary curves. Estimating the contribution from such a curve will require different methods depending on whether it has degree 1, 2 or higher. In the linear case we use Huxley's bounds from \cite{Huxley94}, the proof of which depends on what Huxley calls Swinnerton-Dyer's determinant method. In the quadratic case, the results for Thue inequalities from \S \ref{sec:binary} come into play. Finally, auxiliary curves of higher degree are handled via the bound \eqref{eq:on_curve}. 

\section{The approximate determinant method}
\label{sec:determinant}

Given a non-singular form $F \in \ZZ[x_1,\dotsc,x_n]$ of degree $k \geq 3$ and real numbers $B \geq 1$ and $1\leq \gamma < k$, we consider the set of inequalities
\begin{equation}
\label{eq:n-ary_ineq}
|F(x_1,\dotsc,x_n)| \leq B^\gamma, \quad B/2 < \max_{i=1,\dotsc,n} |x_i| \leq B. 
\end{equation}
Let $S_\gamma(F,B)$ be the set of integer solutions $\xx \in \ZZ^n$ to the inequalities \eqref{eq:n-ary_ineq}. The following theorem may be viewed as a version of \cite[Thm. 14]{Heath-Brown02} for inequalities. It generalizes the results in \cite{Heath-Brown09} and \cite{Marmon11}.

\begin{thm}
\label{thm:approximative_14}
Let $F(x_1,\dotsc,x_n)$ be a non-singular homogeneous polynomial of degree $k\geq 3$. For any $\gamma \in [0,k-n/(n-1))$, put
\[
\theta = \theta(n,k,\gamma) = (n-2)\left(\frac{n}{n-1}\right)^{(n-1)/(n-2)}(k-\gamma)^{-1/(n-2)}. 
\]
Then, for any $\epsi >0$, there is a positive integer $D$, depending only on $k$, $\gamma$ and $\epsi$, and a collection of $O_F(B^{\theta + \epsi})$ forms $A_i$ of degree $D$, such that every point $\xx \in S_\gamma(F,B)$ is a solution to one of the equations $A_i(\xx) = 0$. Moreover, the coefficients of the forms $A_i$ may be chosen as coprime integers of size $O(B^c)$, where $c$ and the implied constant depend only on $k$, $\gamma$ and $\epsi$. 
\end{thm}

To prove Theorem \ref{thm:approximative_14}, we need two lemmas, generalizing \cite[Lemmata 1 and 2]{Heath-Brown09}.
\begin{lemma}
\label{lem:covering}
Let $F(x_1,\dotsc,x_n)$ be a non-singular homogeneous polynomial of degree $k$. There is a natural number $M_0$, depending only on $F$, with the following property: if $[-1,1]^{n-1}$ is covered by smaller $(n-1)$-cubes
\[
\cC_\aa =  [a_1,a_1+(M_0 M)^{-1}]\times \dotsb \times [a_{n-1},a_{n-1}+(M_0 M)^{-1}],
\]
for some positive integer $M$, then the number of such boxes containing a solution $(t_1,\dotsc,t_{n-1})\in \RR^{n-1}$ to the inequality
\begin{equation}
\label{eq:F_ineq}
|F(t_1,\dotsc,t_{n-1},1)| \leq \frac{1}{M_0 M}
\end{equation}
is at most $O(M^{n-2})$. More precisely, if $\cC = \cC_\aa$ is such a box containing a solution to \eqref{eq:F_ineq}, then for some index $i \in \{1,\dotsc,n-1\}$ we have
\[
\left\vert \frac{\partial F}{\partial x_i}(t_1,\dotsc,t_{n-1},1) \right\vert \gg 1
\] 
throughout $\cC$, and there are only $O(1)$ choices of $a'_i$ such that the box $\cC_{\aa'}$ corresponding to $\aa' = (a_1,\dotsc,a_{i-1},a'_i,\dotsc,a_{n-1})$ also contains a solution.

\end{lemma}

\begin{lemma}
\label{lem:implicit}
In the notation of Lemma \ref{lem:covering}, suppose that
\[
\cC=[a_1,a_1+(M_0 M)^{-1}]\times \dotsb \times [a_{n-1},a_{n-1}+(M_0 M)^{-1}]
\]
contains a solution to \eqref{eq:F_ineq}, and that $\left\vert \frac{\partial F}{\partial x_{n-1}}(a_1,\dotsc,a_{n-1},1)\right\vert \gg 1$.

If $(t_1,\dotsc,t_{n-1}) = (a_1+u_1,\dotsc,a_{n-1}+u_{n-1}) \in \cC$, put
\[
w=F(t_1,\dotsc,t_{n-1},1)-F(a_1,\dotsc,a_{n-1},1).
\]
Then there exist, for each $m \in \NN$, two polynomials
\[
\Phi_m(u_1,\dotsc,u_{n-2},w), \quad  \Psi_m(u_1,\dotsc,u_{n-1},w),
\]
such that $\Phi_m$ has no constant term and $\Psi_m$ has no term of degree less than $m$, and such that the relation
\[
u_{n-1} = \Phi_m(u_1,\dotsc,u_{n-2},w) + u_{n-1} \Psi_m(u_1,\dotsc,u_{n-1},w)
\]
holds throughout $\cC$. Moreover, $\Phi_m$ and $\Psi_m$ have degree $O_{m}(1)$ and coefficients of size $O_m(1)$.
\end{lemma}

We omit the proofs of Lemmas \ref{lem:covering} and \ref{lem:implicit}, which may easily be extrapolated from the proofs of the corresponding results in \cite{Heath-Brown09}.

In the proof of Theorem \ref{thm:approximative_14}, we shall need estimates for certain sets of monomials. To this end we define, for any natural number $m$, and any real numbers $\alpha, \nu \in [1,\infty)$, the set 
\[
D_{m,\alpha,\nu} = \left\{(x_1,\dotsc,x_m) \in \ZZ^n: x_1,\dotsc,x_m \geq 0,\ \left(\sum_{i=1}^{m-1} x_i \right) + \alpha x_m \leq \nu\right\}. 
\]
Being a subset of $\ZZ_{\geq 0}^{m}$, we may identify $D_{m,\alpha,\nu}$ with a set of monomials in $m$ variables. The quantities we shall be interested in are then the number of monomials in this set and the accumulated sum of their degrees, \emph{i.e.}
\begin{gather*}
\Sigma_{m,\alpha,\nu} = \sum_{\xx \in D_{m,\alpha,\nu}} 1, \text{ and} \\
\Phi_{m,\alpha,\nu} = \sum_{\xx \in D_{m,\alpha,\nu}} (x_1+\dotsb+x_{m-1}+\alpha x_m).
\end{gather*}
We have $D_{m,\alpha,\nu} = \Delta_{m,\alpha,\nu} \cap \ZZ^m$, where
\[
\Delta_{m,\alpha,\nu} = \left\{(x_1,\dotsc,x_m) \in \RR^n:\ x_1,\dotsc,x_m \geq 0,\ \left(\sum_{i=1}^{m-1} x_i \right) + \alpha x_m \leq \nu \right\} 
\]
is an $m$-simplex in $\RR^n$. We shall obtain estimates for $\Sigma_{m,\alpha,\nu}$ and $\Phi_{m,\alpha,\nu}$ via the corresponding quantities
\begin{gather*}
V_{m,\alpha,\nu} = \mathrm{vol}(\Delta_{m,\alpha,\nu}) = \int_{\Delta_{m,\alpha,\nu}} \mathrm{d}\xx,\\
C_{m,\alpha,\nu} =  \int_{\Delta_{m,\alpha,\nu}} (x_1+\dotsb+x_{m-1}+\alpha x_m) \mathrm{d}\xx.
\end{gather*}

\begin{prop}
\label{prop:volumes}
We have
\begin{gather}
V_{m,\alpha,\nu} = \frac{\nu^m}{\alpha m!}, \label{eq:V} \\
C_{m,\alpha,\nu} = \frac{\nu^{m+1}}{\alpha (m+1)(m-1)!}. \label{eq:C}
\end{gather}
For fixed $m, \alpha$, we have
\begin{gather}
\Sigma_{m,\alpha,\nu} = \frac{\nu^m}{\alpha m!} + O_{m,\alpha} \left(\nu^{m-1}\right) \label{eq:Sigma}, \\
\Phi_{m,\alpha,\nu} = \frac{\nu^{m+1}}{\alpha (m+1)(m-1)!} + O_{m,\alpha} \left(\nu^{m}\right). \label{eq:Phi}
\end{gather}
as $\nu \to \infty$.
\end{prop}

\begin{proof}
It is easily seen that
\[
V_{m,\alpha,\nu} = \frac{\nu^m}{\alpha} V_{m,1,1} \quad \text{and}  \quad C_{m,\alpha,\nu} = \frac{\nu^{m+1}}{\alpha} C_{m,1,1}.
\]
By the variable change $s_j = \sum_{i\leq j} x_i, j = 1,\dotsc,m$, one may easily evaluate $V_{m,1,1}$ and $C_{m,1,1}$ as iterated integrals. Thus, we have
\begin{gather*}
V_{m,1,1} = \int_{0\leq s_1 \leq \dotsb \leq s_m \leq 1} \mathrm{d}\bf{s} = \frac{1}{m!},\\
C_{m,1,1} = \int_{0\leq s_1 \leq \dotsb \leq s_m \leq 1} s_n \mathrm{d}\bf{s} = \frac{1}{(m+1)(m-1)!},
\end{gather*}
and \eqref{eq:V} and \eqref{eq:C} are proven. 

Let $K_{m,\alpha,\nu}$ be the set obtained by attaching a unit $m$-cube to each lattice point in $\Delta_{m,\alpha,\nu}$ in the following fashion:
\[
K_{m,\alpha,\nu} = \bigcup_{\xx \in \Delta_{m,\alpha,\nu} \cap \ZZ^m} [x_1,x_1+1] \times \dotsb \times [x_m,x_m+1]. 
\]
Then we have $\Delta_{m,\alpha,\nu} \subseteq K_{m,\alpha,\nu} \subseteq \Delta_{m,\alpha,\nu+m-1+\alpha}$, whence
\begin{gather*}
V_{m,\alpha,\nu} \leq \Sigma_{m,\alpha,\nu} \leq V_{m,\alpha,\nu+m-1+\alpha}.
\end{gather*}
Inserting \eqref{eq:V}, we obtain the estimate \eqref{eq:Sigma}.

To prove \eqref{eq:Phi}, we put $g(\xx) = x_1 + \dotsb + x_{m-1} + \alpha x_m$ for brevity. If $\xx \in D_{m,\alpha,\nu}$ and $x_i \leq t_i \leq x_i+1$ for $i=1,\dotsc,m$, we have
\[
g(\xx) \leq g(\mathbf{t}) \leq g(\xx) + m-1+\alpha. 
\]
Consequently, on the one hand
\begin{align*}
C_{m,\alpha,\nu} = \int_{\Delta_{m,\alpha,\nu}} g(\mathbf{t}) \mathrm{d} \mathbf{t} &\leq \sum_{\xx \in D_{m,\alpha,\nu}} g(\xx) + (m-1+\alpha) \Sigma_{m,\alpha,\nu} \\
&= \Phi_{m,\alpha,\nu} + O_{m,\alpha}(\nu^m),
\end{align*}
on the other hand,
\[
\Phi_{m,\alpha,\nu} \leq C_{m,\alpha,\nu+m-1+\alpha},
\]
so the estimate \eqref{eq:Phi} follows from \eqref{eq:C}.
\end{proof}

We may now proceed with the proof of Theorem \ref{thm:approximative_14}, along the lines of \cite{Heath-Brown09, Marmon11}. Clearly, it suffices to consider the subset $S'$ of points $\xx \in S_\gamma(F,B)$, where $B/2 < |x_n| = \max_i |x_i| \leq B$. For $\xx \in S'$, we have
\begin{equation*}
|F(t_1,\dotsc,t_{n-1},1)| \leq 2^k B^{\gamma-k} 
\end{equation*}
where $\mathbf{t} = (x_1/x_n,\dotsc,x_{n-1}/x_n)$. Let $M_0$ be as in Lemma \ref{lem:covering}. Then, as soon as
\begin{equation}
\label{eq:M_condition}
M \leq \frac{B^{k-\gamma}}{2^k M_0},
\end{equation}
Lemma \ref{lem:covering} implies that for each $\xx \in S'$, $\mathbf{t}$ belongs to one of $O(M^{n-2})$ boxes of equal sidelength $(M_0M)^{-1}$. We fix such a box 
\[
\cC = [a_1,a_1+(M_0 M)^{-1}]\times \dotsb \times [a_{n-1},a_{n-1}+(M_0 M)^{-1}] 
\]
and let $S'_\cC$ be the set of points $\xx \in S'$ for which $\ttt \in \cC$. Our aim will now be to show that a suitable choice of $M$ implies the existence of a form $A_\cC$ with integer coefficients such that $A_\cC(\xx) = 0$ for all $\xx \in S'_\cC$.

Let $D$ be a natural number and let $f_1,\dotsc,f_s$, where 
\begin{equation}
\label{eq:s}
s = s(D) = \binom{D+n-1}{n-1} = \frac{D^{n-1}}{(n-1)!}, 
\end{equation}
be the monomials in $n$ variables of degree $D$ (in any order). If $S'_\cC = \{\xx^{(i)}: 1\leq i \leq I\}$, it now suffices to prove that the matrix 
\[
X = (f_j(\xx^{(i)}))_{\substack{1\leq i \leq I\\
1 \leq j \leq s}} 
\]
has rank less than $s$, for $D$ large enough. Indeed, one may then find a non-zero vector $\bb \in \ZZ^s$, whose entries have size $O(B^{s^2D})$, such that $X \bb = \bf{0}$. Taking each entry of $\bb$ as coefficient for the corresponding monomial, we obtain a form $A_\cC$ with the desired property. If $J < s$, we trivially have $\mathrm{rank}(X) < s$. Otherwise, we choose a subset of $S'_\cC$ of cardinality $s$ --- without loss of generality we may take $\{\xx^{(1)},\dotsc,\xx^{(s)}\}$ --- and prove that the corresponding $s\times s$-subdeterminant
\[
\Delta_1 = \det(f_j(\xx^{(i)})_{1\leq i,j \leq s}
\]
vanishes. Since the value of $\Delta_1$ is an integer, it suffices to prove that $|\Delta_1| < 1$.

First, note that
\begin{equation}
\label{eq:Delta1}
\Delta_1 = \prod_{j=1}^s (x^{(i)}_n)^D \Delta_2 \ll B^{sD}|\Delta_2|,
\end{equation}
where $\Delta_2 =  \det \big( f_j(t^{(i)}_1,\dotsc,t^{(i)}_{n-1},1) \big)$. By Lemma \ref{lem:covering} we may assume, without loss of generality, that $\left\vert \partial F/\partial x_{n-1}(\aa,1) \right\vert \gg 1$, where $\aa = (a_1,\dotsc,a_{n-1})$. For $\ttt \in \cC$, let $\uu = (u_1,\dotsc,u_{n-1})$ and $w$ be as defined in Lemma \ref{lem:implicit}, and let $\xi = F(\ttt,1) = w+ F(\aa,1)$. Observe that $F(\aa,1) \ll_F M^{-1}$ by the Mean value theorem, and $\xi \ll M^{-1}$ by our assumption \eqref{eq:M_condition}, whence also $w \ll_F M^{-1}$. 

For each monomial $f_i$, write 
\(
f_i(\ttt,1) = G_i(\uu),
\)
where $G_i=G_{i,\cC}$ is a polynomial of degree $D$. If we write $\uu' = (u_1,\dotsc,u_{n-2})$ for short, the conclusion from Lemma \ref{lem:implicit} for any given $m \in \NN$ is now that
\begin{align*}
f_i(\ttt,1) &= G_i\big(\uu',\Phi_m(\uu',w) + u_{n-1} \Psi_m(\uu,w)\big)\\
&= g_i(\uu',\xi) + O_m(M^{-(m+1)}),
\end{align*}
where the $g_i$'s are polynomials of degree $O_m(1)$, depending on $\cC$, but with coefficients bounded in terms of $m$ (and $F$). We conclude that
\begin{equation}
\label{eq:Delta2}
\Delta_2 = \Delta_3 + O_m(M^{-(m+1)}),
\end{equation}
where $\Delta_3 =  \det \big( g_j\big(u^{(i)}_1,\dotsc,u^{(i)}_{n-2},\xi^{(i)}\big) \big)$.

We shall now estimate the determinant $\Delta_3$ using Lemma 3 in \cite{Heath-Brown09}. Thus, let $m_1,\dotsc,m_s$ be monomials in $n-1$ variables, enumerated in such a way that $\Vert m_1 \Vert \geq \Vert m_2 \Vert \geq \dotsb \Vert m_s \Vert$, where
\[
\Vert m \Vert := m\big((M_0 M)^{-1},\dotsc,(M_0 M)^{-1},2^k B^{\gamma-k}\big).
\]
Then, according to \cite[Lemma 3]{Heath-Brown09}, we have
\[
\Delta_3 \ll_{D,k} \prod_{i=1}^s \Vert m_i \Vert \ll_{D,k} M^{-f},
\]
where
\[
f = \sum (e_1+\dotsb+e_{n-2}+\alpha e_{n-1}),
\]
the sum being taken over all $(n-1)$-tuples of non-negative integers $(e_1,\dotsc,e_{n-1})$ such that $x_1^{e_1}\dotsb x_{n-1}^{e_{n-1}}$ occurs among the monomials $m_i$, $1\leq i \leq s$, and with $\alpha$ defined by the relation
\begin{equation}
\label{eq:alpha}
(M_0 M)^\alpha = 2^{-k} B^{k-\gamma}.
\end{equation}
Under the assumption \eqref{eq:M_condition}, we have $\alpha \geq 1$. If $\nu \geq 1$ is defined by $\Vert m_s \Vert = (M_0 M)^\nu$, then we have 
\begin{gather*}
\Sigma_{n-1,\alpha,\nu-1} \leq s \leq \Sigma_{n-1,\alpha,\nu} \quad \text{and} \quad
\Phi_{n-1,\alpha,\nu-1} \leq f \leq \Phi_{n-1,\alpha,\nu},
\end{gather*}
in the notation introduced above. From Proposition \ref{prop:volumes} and \eqref{eq:s}, we infer that $\nu = \alpha^{1/(n-1)} D + O_\alpha(1)$ and
\begin{equation*}
f = \frac{\alpha^{1/(n-1)}D^n}{n(n-2)!} + O_\alpha(D^{n-1}).
\end{equation*}

By \eqref{eq:Delta1}, \eqref{eq:Delta2} and \eqref{eq:alpha}, choosing $m=f \ll_{\alpha,D} 1$, we have
\begin{equation}
\label{eq:Delta1_2}
\Delta_1 \ll_{D,\alpha} B^{s D} M^{-f} \ll_{D,\alpha} B^{\beta(D)},
\end{equation}
where
\begin{align*}
\beta(D) &= s D - \frac{(k-\gamma)f}{\alpha}  = \frac{c D^n}{(n-2)!}  + O_\alpha(D^{n-1}),
\end{align*}
with
\begin{align*} 
c &= c(\alpha) = \frac{1}{n-1}-\frac{k-\gamma}{n \alpha^{(n-2)/(n-1)}}.
\end{align*}
If $\alpha$ is small enough to render $c$ negative, then we will have $\beta(D) < 0$ for $D$ large enough, implying that $|\Delta_1| < 1$ for $B$ large enough, as desired. 

We preliminarily set
\[
\tilde\alpha = (1-\lambda) \left(\frac{(n-1)(k-\gamma)}{n}\right)^{(n-1)/(n-2)},
\]
where $\lambda \in (0,1)$ is a parameter to be chosen properly. Since, by assumption, $\gamma \in [0,k-n/(n-1))$, we have $\tilde\alpha > 1$ for $\lambda$ small enough. As $M$ was required to be an integer in our construction, we then define 
\[
M = \left\lceil \frac{B^{(k-\gamma)/\tilde\alpha}}{2^{k/\tilde\alpha}M_0} \right\rceil. 
\]
If we assume, as we may, that $B > 2^k M_0$, say, then it follows that
\[
\tilde\alpha \leq \alpha \leq \tilde\alpha + O((\log B)^{-1}).
\]
The condition \eqref{eq:M_condition} is then indeed satisfied, at least as soon as $B \gg 1$. Moreover, we have $c(\alpha) \leq c(\tilde\alpha) \leq c_0(\lambda)$, say, where $c_0(\lambda) < 0$. This in turn implies that $\beta(D) < -1$, say, for $D$ greater than some constant $D_0(\lambda)$, so that by \eqref{eq:Delta1_2} we indeed get $|\Delta_1| < 1$ for $B \gg_{F,\lambda} 1$. 

We have now achieved our goal, establishing for each box $\cC$ the existence of a form $A_\cC$ with the desired properties. To finish the proof of Theorem \ref{thm:approximative_14}, we need only bound that the number $\cN$ of boxes needed. We recall that $\cN \ll_F M^{n-2}$, where we now have $M \ll_F B^{(k-\gamma)/\tilde\alpha}$. Thus we get
\begin{align*}
\frac{\log \cN}{\log B} &\leq \frac{(n-2)(k-\gamma)}{\tilde\alpha} + O((\log B)^{-1})\\
&= \theta(1 + O(\lambda)) + O((\log B)^{-1}) \leq \theta + \epsi + O((\log B)^{-1})
\end{align*}
as soon as $\lambda \ll_\epsi 1$, as desired.


\section{The number of solutions to binary homogeneous inequalities}
\label{sec:binary}
Given a binary form $F \in \ZZ[x_1,x_2]$ of degree $d \geq 2$ and a positive integer $P$, we are interested in the set $S_F(P)$ of integer solutions $\xx \in \ZZ^2$ to the inequality
\begin{equation}
\label{eq:thue_ineq}
1 \leq |F(\xx)| \leq P.
\end{equation}
If $d \geq 3$ and $F$ is irreducible over $\QQ$, then this is called a \emph{Thue inequa\-li\-ty}. 
Let $N_F(P)= \#S_F(P)$. Heuristically, one expects that $N_F(P) \sim P^{2/d} A_F$, where $A_F$ is the area of the region in $\RR^2$ defined by $|F(\xx)| \leq 1$. Schmidt proved the upper bound $N_F(P) \ll dP^{2/d}(1 + \log P^{1/d})$, where the implied constant is absolute. Thunder \cite{Thunder01} sharpened this result, proving that $N_F(P) \ll_d P^{2/d}$ under the weaker hypotheses that $A_F$ be finite and $F$ do not vanish at any non-trivial rational point.  

When these hypotheses are not satisfied, for example if $A_F$ is infinite, one may still address the 
problem of estimating the number of solutions to \eqref{eq:thue_ineq} of \emph{bounded height}. For any real number $B\geq 1$, let $S_F(B,P)$ be the set of primitive integer solutions 
$\xx$
to \eqref{eq:thue_ineq} satisfying $\Vert \xx \Vert_{\infty} \leq B$, and let $N_F(B,P) = \#S_F(B,P)$. Our aim in this section is to find a good upper bound for the quantity $N_F(B,P)$ as $B \to \infty$.

We introduce some notation, following \cite{Thunder01}. For a linear form $L(\xx) = a_1 x_1 + a_2 x_2$ we use the corresponding bold letter $\bf L$ to denote its coefficient vector $(a_1,a_2)$. This will be viewed as a row vector, and consequently its transpose, denoted $\bf L^{\tr}$, is a column vector. We also define $\overline{\bf L} = (\overline{a_1},\overline{a_2})$, where $\overline{a}$ denotes the complex conjugate of a complex number $a$. Given a factorization $F(\xx) = \prod_{i=1}^d L_i(\xx)$, where $L_i \in \CC[\xx]$ are linear forms, we define the height of $F$ to be
\[
\cal H(F) := \prod_{i=1}^d \Vert \bf L_i \Vert.
\]
Here and throughout this section, $\Vert \cdot \Vert$ denotes the $L^2$-norm. As is easily seen, the height is independent of the factorization. Moreover, it is shown in \cite[Lemma 1]{Thunder01} that $\cal H(F) \geq 1$ for any form $F \in \ZZ[\xx]$. 

Let $\GL_2(\ZZ)$ be the set of $2\times 2$-matrices with integer entries and determinant $\pm 1$. Two forms $F,G \in \ZZ[\xx]$ are said to be equivalent if $G = F \circ T$ for some $T \in \GL_2(\ZZ)$. The quantities $V_F$ and $N_F(P)$ are clearly preserved under equivalence, whereas $\cal H(F)$ is not.

\begin{def*}
Given a factorization as above, let $a(F)$ be the maximum cardinality of a subset $I \subset \{1,\dotsc,d\}$ such that the vectors $\bf L_i$ for $i \in I$ are pairwise proportional. In other words, $a(F)$ is the highest multiplicity of any complex zero of $F$. This is an easy special case of the quantity $a(F)$ introduced in \cite{Thunder01}.
\end{def*}

We are now ready to state the main result of this section. Given a sublattice $\Lambda \subset \ZZ^2$ of dimension two, we define 
\[
S_F(\Lambda,B,P) = S_F(B,P) \cap \Lambda \quad \text{and} \quad N_F(\Lambda,B,P) = \#S_F(\Lambda,B,P).
\]

\begin{prop}
\label{prop:binary}
Let $F \in \ZZ[\xx]$ be a form of degree $d$, and suppose that $a(F) \leq d/2$. Then the estimate
\[
N_F(\Lambda,B,P) \ll_d \left(\frac{P^{2/d}}{\det(\Lambda)}+1\right)\log B
\]
holds for any two-dimensional sublattice $\Lambda \subset \ZZ^2$. In particular,
\[
N_F(B,P) \ll_d  P^{2/d} \log B.
\]
\end{prop}

Following \cite{Thunder05}, we also introduce the quantity 
\[
\m(F) = \inf \{\cal H(F \circ T)\}, 
\]
where the infimum is taken over matrices $T \in \GL_2(\RR)$ with $\det(T) = \pm 1$. Unlike the height $\cal H(F)$, $\m(F)$ is indeed preserved under equivalence. If $F \in \ZZ[\xx]$ does not vanish at any non-trivial rational point, then one has $\m(F) \geq 2^{-5d/4}$, by \cite[Theorem 1]{Thunder05}. We shall derive a uniform lower bound that is valid in a slightly more general setting.

\begin{prop}
\label{prop:m(F)}
Let $F \in \ZZ[\xx]$ be a form of degree $d$ and suppose that $a(F) \leq d/2$. Then we have $\m(F) \geq 2^{-2d}$. 
\end{prop}

\begin{proof}
For an arbitrary matrix $T \in \GL_2(\RR)$ with $\det(T) = \pm 1$, we define the lattice $\bf M = T^{-1} \ZZ^2$. Let $\lambda_1\leq \lambda_2$ be the successive minima of $\bf M$ with respect to the unit square $\{\xx \in \RR^2: \Vert \xx \Vert_\infty \leq 1\}$, and let $\{\uu_1,\uu_2\}$ be a basis for $\bf M$ such that $\Vert \uu_i \Vert_\infty = \lambda_i$ for $i=1,2$. (It is a well known fact that such a basis exists for lattices in dimension at most four.) By Minkowski's second theorem we have $\lambda_1 \lambda_2 \leq 1$, since $\det(\bf M) = 1$.

Let $S = (T\uu_1^\tr,T\uu_2^\tr) \in \GL_2(\ZZ)$. Then we have
\begin{align*}
\left\Vert \bf L_i S \right\Vert^2 &= \left\Vert \big((\bf L_i T) \uu_1^\tr,(\bf L_i T) \uu_2^\tr\big)\right\Vert^2 \leq 
(\Vert \uu_1 \Vert^2 + \Vert \uu_2 \Vert^2)\Vert \bf L_i T \Vert^2 \\
&\leq 4 \lambda_2^2 \Vert \bf L_i T \Vert^2
\end{align*}
for $i=1,\dotsc,d$, so that 
\[
\cH(F \circ S) = \prod_{i=1}^d  \Vert \bf L_i S \Vert \leq 2^d \lambda_2^d\,\cH(F\circ T).
\]
In particular, since $F\circ S \in \ZZ[\xx]$, we get
\begin{equation}
\label{eq:height_lower_bound}
1 \leq  2^{d} \lambda_2^d\,\cH(F\circ T).
\end{equation}

We shall now derive an upper bound for $\lambda_2$. Let $\xx_1^\tr,\xx_2^\tr$ be the columns of $T$, and let $\zz_1^\tr,\zz_2^\tr$ be the columns of $S$. 
Note that $\bf L_i T = (L_i(\xx_1), L_i(\xx_2))$. 
We shall now distinguish between two cases. Suppose first that $F(\zz_1) \neq 0$. Writing $\uu_1 = (u_{11},u_{12})$, we get
\begin{align*}
1 & \leq |F(\zz_1)| = \prod_{i=1}^d |L_i(\zz_1)| = \prod_{i=1}^d |u_{11} L_i(\xx_1) + u_{12} L_i(\xx_2)| \\
&= \prod_{i=1}^d |(\bf L_i T) \uu_1^\tr| \leq \prod_{i=1}^d \Vert \uu_1 \Vert \Vert \bf L_i T \Vert = (\sqrt 2 \lambda_1)^d \cH(F\circ T).  
\end{align*}
From this inequality it follows that
\[
\lambda_2 \leq \frac{1}{\lambda_1} \leq \sqrt{2}\ \cH(F\circ T)^{1/d}.
\]
Inserting this bound into \eqref{eq:height_lower_bound}, we get
\[
\cH(F\circ T) \geq 2^{-3d/4}. 
\]

Suppose instead that $F(\zz_1) = 0$. If $\zz_1 = (z_{11},z_{12})$, we may then write $F(\xx) = L(\xx)^a G(\xx)$, where $L(\xx) = z_{12} x_1 - z_{11} x_2$, and where $G \in \ZZ[\xx]$ satisfies $G(\zz_1) \neq 0$. By the same arguments as above we then have
\[
1 \leq |G(\zz_1)| \leq  (\sqrt 2 \lambda_1)^{d-a} \cH(G\circ T).
\]
Furthermore, writing $\xx_i = (x_{i1},x_{i2})$ for $i=1,2$, we have
\begin{align*}
\uu_1 &= \zz_1(T^{-1})^\tr = \pm (x_{22} z_{11} - x_{21} z_{12}, -x_{12} z_{11} + x_{11} z_{12}) \\
&= \pm (-L(\xx_2),L(\xx_1)),
\end{align*}
whence
\begin{align*}
\lambda_1 &= \Vert \uu_1 \Vert_\infty \leq \Vert \uu_1 \Vert = \Vert (-L(\xx_2),L(\xx_1)) \Vert 
= \Vert (L(\xx_1),L(\xx_2)) \Vert \\ 
&= \cH(L\circ T).
\end{align*}
Combining these two estimates, we get
\[
\cH(F\circ T) = \cH(L\circ T)^a \cH(G\circ T) \geq \lambda_1^a (\sqrt{2} \lambda_1)^{a-d} = 2^{(a-d)/2} \lambda_1^{2a-d}. 
\]
In case $a=d/2$, this immediately implies that $\cH(F\circ T) \geq 2^{-d/4}$, which is certainly sufficient. If $a < d/2$, we get 
\[
\lambda_2 \leq \frac{1}{\lambda_1} \leq 2^{(d-a)/(2(d-2a))} \cH(F\circ T)^{1/(d-2a)}. 
\]
As before we insert this bound into \eqref{eq:height_lower_bound}, and obtain
\[
1 \leq 2^{d(2d-3a)/(2(d-2a))} \cH(F\circ T)^{2(d-a)/(d-2a)}, 
\]
which in turn gives 
\[
\cH(F \circ T) \geq 2^{-d(2d-3a)/(d-a)}.
\]
For $1 \leq a < d/2$, we have $d(2d-3a)/(d-a) \leq 2d$, which proves the desired estimate.
\end{proof}

\begin{lemma}
\label{lem:Thunder05_3}
Let $F \in \ZZ[\xx]$ be a form of degree $d$ and suppose that $a(F) \leq d/2$. Choose a factorization $F(\xx) = \prod_{i=1}^d L_i(\xx)$, where $\bf L_i \in \RR^2$ for $i=1,\dotsc,r$, $\bf L_i \in \CC^2$ for $i=r+1,\dotsc,r+2s = d$ and $\bf L_{i+s} = \overline{\bf L_i}$ for $i = r+1,\dotsc,r+s$. Then we have
\[
\sum_{1\leq i_1,i_2 \leq d} \left|\det(\bf L_{i_1}^{\tr},\bf L_{i_2}^{\tr})\right|^2 \geq \frac{{d}^2}{2} \m(F)^{4/d}. 
\]
\end{lemma}

\begin{proof}
This follows from the proof of \cite[Lemma 3]{Thunder05}. Indeed, our assumption about $a(F)$ certainly suffices to establish the existence of two linearly independent coefficient vectors $\bf L_i$, so we may replace the hypothesis in \cite{Thunder05} that $A_F$ be finite with this weaker hypothesis.
\end{proof}

\begin{lemma}
\label{lem:Thunder01_6}
Suppose that $a(F) \leq d/2$. Then, for each $\zz \in \RR^2$, there are $1\leq i_1,i_2 \leq d$ such that $\bf L_{i_1}$ and $\bf L_{i_2}$ are linearly independent and 
\[
\frac{\left|L_{i_1}(\zz)\right|\left|L_{i_2}(\zz)\right|}{\left|\det(\bf L_{i_1}^{\tr},\bf L_{i_2}^{\tr})\right|} \ll \frac{|F(\zz)|^{2/d}}{\m(F)^{2/d}}. 
\]
\end{lemma}

\begin{proof}
If $F(\zz) = 0$, there is nothing to prove, so let us assume other\-wise. As both sides in the sought inequality are independent of the particular factorization used, let us initially choose a factorization  $F(\xx) = \prod_{i=1}^d L_i(\xx)$ as in the statement of Lemma \ref{lem:Thunder05_3}. Putting
\[
L_i'(\xx) := \frac{|F(\zz)|^{1/d}}{|L_i(\zz)|}L_i(\xx)
\]
for $i=1,\dotsc,d$, we get an alternative factorization $F(\xx) = \prod_{i=1}^d L_i'(\xx)$ that still satisfies the hypotheses of Lemma \ref{lem:Thunder05_3}. It follows from the lemma that
\[
\sum_{1\leq i_1,i_2 \leq d} \left|\det\left(\left(\bf L_{i_1}'\right)^{\tr},\left(\bf L_{i_2}'\right)^{\tr}\right)\right|^2 \geq \frac{{d}^2}{2} \m(F)^{4/d}. 
\]
The sum on the left hand side has at most $d(d-1)$ non-zero terms, so for some pair $(i_1,i_2)$ we must have
\[
\left|\det\left(\left(\bf L_{i_1}'\right)^{\tr},\left(\bf L_{i_2}'\right)^{\tr}\right)\right|^2 \geq \frac{d}{2(d-1)} \m(F)^{4/d}. 
\]
On the other hand,
\[
\det\left(\left(\bf L_{i_1}'\right)^{\tr},\left(\bf L_{i_2}'\right)^{\tr}\right) = \frac{|F(\zz)|^{2/d}\det\left(\left(\bf L_{i_1}\right)^{\tr},\left(\bf L_{i_2}\right)^{\tr}\right)}{\left|L_{i_1}(\zz)\right|\left|L_{i_2}(\zz)\right|}, 
\]
whence
\[
\frac{\left|L_{i_1}(\zz)\right|\left|L_{i_2}(\zz)\right|}{\det\left(\left(\bf L_{i_1}\right)^{\tr},\left(\bf L_{i_2}\right)^{\tr}\right)}
\leq \left(\frac{2(d-1)}{d}\right)^{1/2} \frac{|F(\zz)|^{2/d}}{\m(F)^{2/d}}.
\]

\end{proof}

\begin{proof}[Proof of Proposition \ref{prop:binary}]
By Lemma \ref{lem:Thunder01_6} and Proposition \ref{prop:m(F)}, each $\xx \in S_F(P)$ satisfies an inequality
\begin{equation}
\label{eq:two_linear_forms}
\frac{\left|L_{i_1}(\xx)\right|\left|L_{i_2}(\xx)\right|}{\left|\det(\bf L_{i_1}^{\tr},\bf L_{i_2}^{\tr})\right|} \ll \frac{P^{2/d}}{\m(F)^{2/d}} \ll P^{2/d}. 
\end{equation}
There are at most $d(d-1)/2$ such inequalities to consider, and by \cite[Lemma 7$'$]{Thunder01}, the solutions to \eqref{eq:two_linear_forms} with $\Vert \xx \Vert \leq B$ lie in the union of $O(\log B)$ convex sets $\cC$ of the type
\[
\{\xx \in \RR^2: |K_1(\xx)| \leq B_1, |K_2(\xx)| \leq B_2\}, 
\]
where $K_i \in \CC[\xx]$ are linear forms satisfying $|\det(\bf K_1^{\tr}, \bf K_2^\tr)| = 1$ and $\Vert K_i \Vert = 1$, and where $B_1 B_2 \ll P^{2/d}$. Lemmas 8 and 9 in \cite{Thunder01} give estimates for the number of lattice points in such convex sets. Indeed, in case $\cC$ contains two linearly independent vectors in $\Lambda$, it follows from \cite[Lemma 8]{Thunder01} and the proof of \cite[Lemma 9]{Thunder01} that 
\[
\#(\cC \cap \Lambda) \ll \frac{P^{2/d}}{\det(\Lambda)}. 
\]
In the opposite case, $\cC$ obviously contributes at most two primitive points to $S_F(\Lambda,B,P)$, so collecting the contributions from all the sets $\cC$ gives the asserted bound.
\end{proof}


\section{Rational points close to curves}
\label{sec:proof}

We restate the case $n=3$ of Theorem \ref{thm:approximative_14} in a slightly different form. 
\begin{cor}
\label{cor:determinant}
There is a positive integer $D$, depending only on $\gamma$ and $\epsi$, a positive integer $M$ satisfying
\begin{equation}
\label{eq:M}
B^{9/4(k-\gamma)} \ll_{F,\gamma,\epsi} M \ll_{F,\gamma,\epsi} B^{9/4(k-\gamma)+\epsi},
\end{equation}
a collection of points $(r_1,s_1),\dotsc,(r_M,s_M) \in [-1,1]^2$, and a collection of $M' \ll M$ irreducible forms $A_1,\dotsc,A_{M'} \in \ZZ[x,y,z]$ of degree at most $D$, such that the following holds. For each primitive solution to \eqref{eq:ternary_ineq} with \eqref{eq:z_biggest}, there is some $1 \leq i \leq M$ for which
\begin{equation}
\label{eq:auxiliary+patch}
0 \leq \frac{x}{z} - r_i \leq M^{-1}, \quad 0 \leq \frac{y}{z} - s_i \leq M^{-1} \quad \text{and} \quad A_i(x,y,z) = 0.
\end{equation}
Moreover, the coefficients of $A_i$ may be chosen as coprime integers of size bounded by $B^c$ for some constant $c$. The points $(r_i,s_i)$ can be chosen as 
\begin{equation}
\label{eq:rational_midpoints}
(r_i,s_i) = \left(\frac{v_i}{M},\frac{w_i}{M}\right)
\end{equation}
for suitable integers $v_i,w_i$.
\end{cor}

Here we have retained the connection between each auxiliary form and its corresponding patch, which is obvious from the proof of Theorem \ref{thm:approximative_14}. Furthermore we now assume our auxiliary forms to be irreducible simply by considering the irreducible factors of the original forms, at the cost of multiplying $M$ by a constant. The bound on the coefficients of the $A_i$ is still valid by Gelfond's lemma \cite[Lemma 1.6.11]{Bombieri_Gubler}. The last statement of the corollary is obvious.

We shall now go into the details of translating between solutions to the inequality \eqref{eq:F_ineq} and rational points close to the affine curve defined by $g(x,y)=0$, where $g(x,y) = F(x,y,1)$. Let us first record a number of useful results. The first one is a sublevel set estimate for functions in one variable that appears in more general form in \cite[Prop. 2.1]{Carbery_Christ_Wright}.

\begin{lemma}
\label{lem:russian}
Let $I \subset \RR$ be an open interval and let $h \in C^{k}(I)$ for some $k \geq 1$. If $|h^{(k)}(x)| \geq C >0$ for all $x \in I$, then
\[
|\{x \in I : |h(x)| \geq \delta\}| \leq 2e\,((k+1)!)^{1/k}\left(\frac{\delta}{C}\right)^{1/k}.
\]
\end{lemma}

The next few lemmas concern algebraic functions in one variable.
\begin{lemma}
\label{lem:subdivision1}
Let $g(x,y) \in \RR[x,y]$ be an absolutely irreducible polynomial of degree $k \geq 2$. There is a subdivision of the unit square $[-1,1]^2$ into $O_k(1)$ smaller squares $S$ such that the following holds:

For each square $S$, the intersection of the curve $g(x,y) = 0$ with $S$ is the graph $y = f(x)$ or $x = f(y)$ of a function $f \in C^\infty(I)$, where $I$ denotes the projection of $S$ onto the $x$- or $y$-axis, accordingly.
The function $f$ satisfies 
\begin{equation}
\label{eq:f'_bounded}
\max\{|f(x)|,|f'(x)|\} \leq 1 \qquad \text{for all } x \in I.
\end{equation}

If $g$ is non-singular, then for each $\ell \geq 2$ we also have that
\begin{equation}
\label{eq:f_ell_bounded}
|f^{(\ell)}(x)| \leq A_\ell \qquad \text{for all } x \in I,
\end{equation}
where $A_\ell$ is a constant depending only on $g$ and $\ell$.
\end{lemma}

\begin{proof}
To achieve \eqref{eq:f'_bounded}, it is enough to subdivide by drawing horizontal and vertical lines through all the at most $O_k(1)$ points $(x,y)$ lying on the curve $g(x,y)=0$ where either $g_x = 0$, $g_y = 0$ or $g_x = \pm g_y$ (with $g_x := \partial g/\partial x$ and $g_y := \partial g/\partial y$) and choosing the appropriate projection onto a coordinate axis for each resulting arc of the curve.

Having done so, assume that we are in the case $y=f(x), x \in I$, and that $g$ is non-singular. It follows by repeated implicit differentiation that
\[
f'(x) = -\frac{g_x(x,y)}{g_y(x,y)}
\]
and in general
\[
f^{(\ell)}(x) = \frac{P_\ell(x,y)}{g_y(x,y)^{2 \ell - 1}}
\]
for some polynomial $P_\ell$ of degree less than $2 k \ell$. By \eqref{eq:f'_bounded} we have
\[
|g_y(x,f(x))| \geq  \min_{\substack{(x,y) \in [-1,1]^2 \\ g(x,y) = 0}} \max \{|g_x(x,y)|,|g_y(x,y)|\} =: c_g, 
\]
for $x \in I$, where $c_g >0$ since $g$ is assumed to be non-singular. Furthermore, $|P_\ell(x,y)|$ is clearly bounded from above by some constant $C_{g,\ell}$ for all $(x,y) \in [-1,1]^2$, so we obtain \eqref{eq:f_ell_bounded} with $A_\ell = C_{g,\ell}/c_g^{2\ell-1}$.
\end{proof}

\begin{lemma}
\label{lem:nonvanishing_derivative}
Let $g(x,y) \in \RR[x,y]$ be an absolutely irreducible polynomial of degree $d \geq 2$, let $I \subset \RR$ be an interval, and suppose that $h \in C^\infty(I)$ satisfies $g(x,h(x))= 0$ on $I$. Then, for every $a \in I$, there is some $1 \leq e \leq d$ for which
\[
h^{(e)}(a) \neq 0.
\]
\end{lemma}

\begin{proof}
By repeated implicit differentiation of the identity $g(x,h) = 0$ one sees that 
\[
\frac{\partial^\ell g}{\partial x^\ell} + \sum_{1 \leq \alpha_1 + \dotsb + \alpha_\ell \leq \ell} C_{\boldsymbol{\alpha}}(x,h) (h')^{\alpha_1} \dotsb (h^{(\ell)})^{\alpha_\ell} = 0,
\]
for any $\ell \geq 1$ and suitable polynomials $C_{\boldsymbol{\alpha}}(x,y)$. If we would have $h'(a) = \dotsb = h^{(d)}(a) = 0$, then the above identity applied for $\ell = d$ would imply that
\[
\frac{\partial^d g}{\partial x^d}(a) = 0.
\]
However, $\frac{\partial^d g}{\partial x^d}$ is a non-zero constant polynomial by the irreducibility assumption, yielding a contradiction.
\end{proof}

We import the following lemma from Bombieri and Pila \cite[Lemma 6]{Bombieri-Pila}:

\begin{lemma}
\label{lem:BP_Lemma6}
Let $g(x,y) \in \RR[x,y]$ be an absolutely irreducible polynomial of degree $d \geq 2$, let $I \subset \RR$ be an interval, and suppose that $f \in C^\infty(I)$ satisfies $g(x,f(x))= 0$ on $I$. Let $C_\ell$ be positive real numbers for $\ell = 1,\dotsc,K$. Then we can divide $I$ into at most $2d^2 K^2$ subintervals $I_\nu$ such that for each $\ell = 1,\dotsc,K$, one of the following holds:
\begin{align}
\label{eq:small_deriv}
|f^{(\ell)}(x)| &\leq C_\ell \quad \text{for all } x \in I_\nu,\\
\label{eq:large_deriv}
|f^{(\ell)}(x)| &\geq C_\ell \quad \text{for all } x \in I_\nu.
\end{align}
\end{lemma}

In the situation under consideration, where $g(x,y) = F(x,y,1)$ for the non-singular (and hence absolutely irreducible) form $F$, an application of Lemma \ref{lem:covering} with $M = O(1)$ allows us to adapt the subdivision in Lemma \ref{lem:subdivision1} so that the following  holds:

If the square $S \subseteq [-1,1]^2$ contains a solution $(x/z,y/z)$ to  \eqref{eq:ternary_ineq} and \eqref{eq:z_biggest}, then one has 
\begin{equation}
\label{eq:partial_lower}
\partial_i F(x,y,1) \gg 1
\end{equation}
throughout $S$, where $i=2$ if $g(x,f(x)) = 0$ and $i=1$ if the other alternative $g(f(y),y)=0$ holds.

Without loss of generality, we restrict our attention to the first case $g(x,f(x)) =0$ from now on. We note that any solution to \eqref{eq:ternary_ineq} and \eqref{eq:z_biggest} with $(x/z,y/z) \in S$ satisfies
\[
\left|\frac{y}{z}-f\left(\frac{x}{z}\right)\right| \ll_F B^{\gamma-k},
\]
assuming that $g(x,f(x)) = 0$ in $S$, with the obvious modification if the other alternative $g(f(y),y)=0$ holds.
Indeed, put $(r,s) = (x/z,y/z)$ for short and observe that 
\begin{align*}
B^{\gamma-k} &\gg |F(r,s,1)| = |g(r,s)| = |g(r,s)-g(r,f(r))| \\
&= |\partial_2 F(r,s,1)| |s-f(r)|,
\end{align*}
so the claim follows from \eqref{eq:partial_lower}.

Given an interval $I$ as in Lemma \ref{lem:subdivision1}, our aim is now to count rational points $(x,y) = (a/q,b/q)$, with $x \in I$ and $B/2 \leq q \leq B$, satisfying
\begin{equation}
\label{eq:close}
\left| \frac{b}{q} - f\left(\frac{a}{q}\right)\right| \leq \frac{1}{B^{\tau}}.
\end{equation}
Recall that $\tau$ was defined in \eqref{eq:tau_def}.
Denote the number of such rational points, excluding those lying on the tangent lines $T_i$,  by 
\begin{equation}
\label{eq:N_I_def}
N_I = N_{f,\tau,I}(B).
\end{equation}

\subsection*{Contribution from lines}
We begin by treating the contribution from forms $A_i$ of degree 1 in Corollary \ref{cor:determinant}. 
Recall that for each form $A_i$, only points satisfying 
\begin{equation}
\label{eq:patch}
r_i \leq x < r_i+M^{-1} 
\end{equation}
need to be taken into account, where $M$ is as in \eqref{eq:M}.

We will use estimates from Huxley \cite{Huxley94}, where we take
\[
\delta = B^{1-\tau},
\]
as implied by \eqref{eq:close}. In Huxley's notation, we further take $M=T=\Delta= 1$ and $Q=B$. The estimate (1.4) in \cite[Theorem 1]{Huxley94} then states that if
\begin{equation}
\label{eq:huxley_fbis}
1/C \leq |f''(x)| \leq C
\end{equation}
on the interval $I$, then 
\begin{equation}
\label{eq:huxley1.4}
N_I \ll C \delta^{1/2} B^2 + C^{1/3}B,
\end{equation}
in the notation introduced above. A careful analysis of the arguments yields that if \eqref{eq:huxley_fbis} is replaced by a more precise assumption 
\begin{equation}
\label{eq:fbis_precise}
C_1 \leq |f''(x)| \leq C_2
\end{equation}
for some constants $C_1 \leq  C_2$, then the following version of \eqref{eq:huxley1.4} holds:
\begin{equation}
\label{eq:huxley1.4'}
N_I \ll \left(\frac{1}{C_1} + \sqrt{\frac{C_2}{C_1}}\right)\delta^{1/2} B^2 + \frac{1}{C_1^{1/3}} B.
\end{equation}
\begin{rem}
\label{rem:majorminor}
Note that if the rational points $(x_i,y_i)$ counted by $N_I$ are ordered by increasing $x$-coordinate $x_1<x_2<\dotsb$, then the first term in \eqref{eq:huxley1.4'} 
corresponds to points lying on 'major sides', that is, line segments containing three or more consecutive points $(x_i,y_i)$, whereas the second term counts the remaining points, lying on 'minor sides'. 
\end{rem}

Thus 
we wish to count rational points satisfying \eqref{eq:close} and in addition lying on a fixed rational line
\begin{equation}
\label{eq:ratline}
\ell x + m y + n = 0,
\end{equation}
where $\ell,m,n$ are integers with $(\ell,m,n) = 1$. We may assume that $I$ is contained in the interval \eqref{eq:patch}. By a subdivision as in Lemma \ref{lem:BP_Lemma6}, we may further suppose that either 
\begin{align}
\label{eq:fbis_small}
|f''(x)| &\leq 1/C \text{ for all } x \in I
\end{align}
or
\begin{align}
\label{eq:fbis_big}
|f''(x)| &\geq 1/C \text{ for all } x \in I,
\end{align}
for some constant $C>0$ to be chosen at a later stage. 

There is no loss of generality in assuming that $m \geq 0$ in \eqref{eq:ratline}, and in fact we must then have 
\[
m>0,
\]
by our assumption that $\gamma < k-2$. If we enumerate the rational points 
\[
(x_i,y_i) = \left(\frac{a_i}{q_i},\frac{b_i}{q_i}\right)
\]
satisfying these conditions as $(x_1,y_1), \dotsc,(x_R,y_R)$, where $x_1<x_2<\dotsb<x_R$, then by \cite[Lemma 5]{Huxley94} we have
\begin{equation}
\label{eq:msep}
x_{i+1} - x_i \geq \frac{m}{q_i q_{i+1}},
\end{equation}
so that 
\begin{equation}
\label{eq:pointsonline}
R \ll 1+\frac{B^2 \lambda}{m},
\end{equation}
where $\lambda := x_R-x_1$.

First we consider an interval $I$ where \eqref{eq:fbis_small} holds. By taking $C$ greater than some constant $C_0$ depending only on $F$, we may assume that each interval of this type contains exactly one point $\alpha$ such that $f''(\alpha) = 0$. 
By Lemma \ref{lem:nonvanishing_derivative} we have $f^{(d)}(\alpha) \neq 0$ for some $3 \leq d \leq k$. Suppose that $d = d_I$ is the minimal such choice, so that $f^{(j)}(\alpha) = 0$ for $2 \leq j \leq d-1$. Clearly we have $d \leq \nu_F$ for the quantity $\nu_F$ defined in the introduction. We now have the following result, where we recall that the quantity $N_I$ was defined in \eqref{eq:N_I_def}.

\begin{lemma}
\label{lem:roth_case}
Consider an interval $I$ where \eqref{eq:fbis_small} holds. Under the assumption \eqref{inflection_assumptions}, the contribution from each line \eqref{eq:ratline} to $N_I$ is 
\[
\ll 1 + B^{2+\epsi} C^{-2} + B^{2+\epsi-\tau/2} M^{-1/2}.
\]
More generally, the contribution is
\[
\ll 1 + B^{2+\epsi} L^{(d+1)/2} + B^{2+\epsi-\tau/2} L^{1/2},
\]
with $d=d_I$ as defined above, where 
\[
L \ll \min\{C^{-1/(d-2)},M^{-1}\}.
\]
\end{lemma}

\begin{proof}
With $\alpha$ as above, let $\beta = f(\alpha)$ and $\kappa = f'(\alpha)$. Note that $\alpha,\beta,\kappa$ are algebraic numbers.  Since there are only finitely many choices for $\alpha$, corresponding to the finitely many inflection points of the algebraic curve, we may in fact bound $f^{(d)}(\alpha)$ from above and below by constants depending only on the form $F$. 


By \eqref{eq:patch}, \eqref{eq:fbis_small} and Lemma \ref{lem:russian} we have
\[
|I| \ll L:= \min\{C^{-1/(d-2)},M^{-1}\}.
\]
For $i=1,R$, we may now write
\begin{align*}
y_i &= f(x_i) + O(B^{-\tau}) \\
&= \beta + \kappa (x_i-\alpha) + O(|x_i-\alpha|^d) + O(B^{-\tau}).
\end{align*}
Using the bound $|x_i -\alpha| \ll L$,
this means that the slope of the line \eqref{eq:ratline} satisfies
\begin{equation}
\label{eq:slopeappr}
-\frac{\ell}{m}-\kappa = \frac{y_R-y_1-\kappa(x_R-x_1)}{x_R-x_1} \ll \frac{L^{d} + B^{-\tau}}{\lambda},
\end{equation}
where as before $\lambda = x_R-x_1>0$.
On the other hand, if $\kappa$ is irrational, then by Roth's theorem \cite[Thm V.2A]{Schmidt80}
we also have
\[
\left|\frac{\ell}{m}+\kappa\right| \gg_\epsi m^{-2-\epsi} 
\]
for any $\epsi >0$. The same bound also holds if $\kappa$ is rational, but $\kappa \neq -\ell/m$. Inserting this bound into \eqref{eq:slopeappr}, we get
\[
\frac{\lambda}{m} \ll m^{1+\epsi}(L^{d} + B^{-\tau}). 
\]
Taking the alternative bound $\lambda \ll L$ into account, we thus have
\begin{align*}
\frac{\lambda}{m} &\ll B^\epsi \min\left\{\frac{L}{m},mL^{d} + mB^{-\tau}\right\} \\
&\ll B^\epsi \left(L^{(d+1)/2} + B^{-\tau/2}L^{1/2}\right).
\end{align*}
Inserting this into \eqref{eq:pointsonline} gives
\begin{equation}
\label{eq:inserting}
\begin{split}
R &\ll 1+ B^{2+\epsi} \left(L^{(d+1)/2} + B^{-\tau/2}L^{1/2}\right)\\
&\ll 1 + B^{2+\epsi} L^{(d+1)/2} + B^{2+\epsi-\tau/2}L^{1/2}.
\end{split}
\end{equation}

We now consider the case where $\kappa$ is rational and $\kappa = -\ell/m$. But in this case we have
\[
-\frac{n}{m} -\frac{\ell}{m} x_1 = y_1 = \beta - \frac{\ell}{m}(x_1-\alpha) + O(|x_1-\alpha|^d) + O(B^{-\tau}).
\]
So either the line \eqref{eq:ratline} coincides with the tangent line at $(\alpha,f(\alpha))$, in which case it is excluded from consideration, or we may again use Roth's theorem to conclude that
\[
m^{-2-\epsi} \ll \left|\frac{n}{m}+(\beta-\frac{\ell}{m}\alpha)\right| \ll L^{d} + B^{-\tau}.
\]
This gives
\begin{align*}
\frac{\lambda}{m} &\ll \lambda m B^\epsi (L^{d} + B^{-\tau}),
\end{align*}
so the resulting bound for $R$ in this case is subsumed in the bound \eqref{eq:inserting}.

The assumption \eqref{inflection_assumptions} means that $d = 3$ in the above considerations. Thus we have $L \leq C^{-1}$ and \eqref{eq:inserting} implies
\[
R \ll 1 + B^{2+\epsi}C^{-2} + B^{2+\epsi-\tau/2}M^{-1/2}.
\]
\end{proof}

Suppose next that the alternative \eqref{eq:fbis_big} holds. 
By Lemma \ref{lem:subdivision1} we still have $|f''(x)| \ll_F 1$ on $I$, so we can divide $I$ further into $O(\log C)$ subintervals, on each of which $f''$ satisifes the bound
\begin{equation}
\label{eq:dyadicfbis}
\frac{1}{D} \leq |f''(x)| \leq \frac{2}{D}
\end{equation}
for some $1 \ll D \leq C$. 

Since each minor side, in the notation of Remark \ref{rem:majorminor}, contributes only $O(1)$ points to $N_I$. Since at most $M$ lines are considered in total, the overall contribution from minor sides is $O(M)$. Thus we may focus on the contribution from major sides. To estimate this, we shall use the duality in \cite[Section 3]{Huxley94}. Let $h(y)$ be the inverse function of $f'(x)$ on the interval $I$. Put
\[
g(y) = yh(y)-f(h(y))
\]
so that $g'(y) = h(y)$ and
\[
g''(y) = h'(y) = \frac{1}{f''(h(y))}. 
\]
Let $\cL$ be a rational line as in \eqref{eq:ratline}. The set of points such that $(x,y) \in \cL$ and 
\begin{equation}
\label{eq:close'}
|y-f(x)|\leq \frac{\delta}B
\end{equation} 
consists of at most two disjoint line segments, as observed in \cite[Lemma 11]{Huxley94}. In fact, a careful inspection of the proof of \cite[Lemma 11]{Huxley94} gives the following.

\begin{lemma}
\label{lem:huxleyL11}
Let $I$ be an interval on which $f$ satisfies \eqref{eq:dyadicfbis}. Let $J \subseteq I$ be an interval such that \eqref{eq:close'} holds for all $(x,y) \in \cL$ with $x \in J$. Suppose that there exist two rational points $(x_i,y_i) = (a_i/q_i,b_i/q_i)$, $i=1,2$,  such that $(x_i,y_i) \in \cL$ and $x_i \in J$. If 
\[
\lambda := x_2-x_1 >0,
\]
then there exists $\xi \in J$ with
\begin{equation}
\label{eq:huxley3.2}
\left|\frac{\ell}{m}+f'(\xi)\right| \leq \frac{2\delta}{\lambda B},
\end{equation}
and we have
\begin{equation}
\label{eq:huxley3.3}
\left|g\left(-\frac{\ell}{m}\right)-\frac{n}{m}\right| \leq \frac{\delta}{B} + \frac{4D\delta^2}{\lambda^2B^2},
\end{equation}
where $D$ is the quantity introduced in \eqref{eq:dyadicfbis}.
\end{lemma}

Each line $\cL$ corresponding to a major side of length $\lambda$ corresponds, by Lemma \ref{lem:huxleyL11}, to a rational point $(x,y) = (\ell/m,n/m)$ close to the dual curve $y=g(x)$, the distance being a decreasing function of $\lambda$. We shall count  those rational points using the bound \eqref{eq:huxley1.4'} above. First we observe that $x = \ell/m$ belongs to an interval of length $O(1)$. Indeed, by \eqref{eq:huxley3.2}, $\ell/m$ lies within a distance $2\delta/(\lambda B)$ of a value of $f'(x)$, which in turn takes values in an interval of length $O(1)$. But
\[
\frac{\delta}{\lambda B} \ll \frac{B^{-\tau}}{\lambda} \ll B^{2-\tau} \ll 1,
\] 
since we always have $\lambda \gg B^{-2}$.
(Our assumption that $\tau \geq 2$ simplifies our treatment compared to Huxley's, where the scaling factor $M$ comes into play.)

Let us count major sides in the range
\begin{equation}
\label{eq:major_range}
\Lambda \leq \lambda < 2 \Lambda, \quad B' \leq m < 2B'.
\end{equation}
We only need to consider values
\begin{equation}
\label{eq:Lambdabound}
B^{-2} \ll \Lambda \ll \mu:= \sqrt{\frac{D\delta}{B}}.
\end{equation}
The upper bound is provided by \cite[Lemma 6]{Huxley94}. Indeed, if $R$ denotes the total number of points on the major side, then we may divide it into two halves, each of which has at least $(R+1)/2$ points, and the shorter of which has length at most $\mu$ by \cite[Lemma 6]{Huxley94}. The contribution to $N_I$ from each major side will be 
\[
\ll \frac{B^2\Lambda}{B'}
\]
by \eqref{eq:msep}, so we may assume that $B'\ll \Lambda B^2$. Let $\delta'$ be given by
\[
\frac{\delta'}{B'} = \frac{\delta}{B} + \frac{4D\delta^2}{\Lambda^2B^2},
\]
where the right hand side comes from \eqref{eq:huxley3.3}. By the upper bound \eqref{eq:Lambdabound}, we in fact have
\begin{equation}
\label{eq:delta'}
\frac{\delta'}{B'} \asymp \frac{D\delta^2}{\Lambda^2B^2}.
\end{equation}
 
Since $D/2 \leq |g''(x)| \leq D$, the bound \eqref{eq:huxley1.4'} 
implies that we have
\[
\ll \delta'^{1/2} B'^2 + B'
\]
choices for $\ell,m,n$. 
At this point we remind ourselves that we are only seeking to estimate the contribution $N_{I,1}$, say, to $N_{I}$ from those at most $O(M)$ lines that arose in the application of Corollary \ref{cor:determinant}. 
The total contribution to $N_{I,1}$ from major sides in the range \eqref{eq:major_range} is therefore
\begin{align*}
& \ll \frac{\Lambda B^2}{B'} \min\left\{
M,\left(\delta'^{1/2} B'^2 + B'\right)\right\}
\\
& \ll  \Lambda B^2 + 
\frac{\Lambda B^2}{B'} \min\left\{
M,\left(\frac{\delta'}{B'}\right)^{1/2} B'^{5/2}\right\} \\
&\ll \Lambda B^2 + 
\frac{\Lambda B^{2}}{B'} M^{3/5}\left(\left(\frac{\delta'}{B'}\right)^{1/2} B'^{5/2}\right)^{2/5}\\
&\ll \Lambda B^2 + 
\Lambda B^{2} M^{3/5}
\left(\frac{\delta'}{B'}\right)^{1/5}. 
\end{align*}
Inserting \eqref{eq:delta'}, we note that $\Lambda$ occurs to a positive power in the resulting expression, whereas $B'$ does not occur at all. Summing over dyadic ranges for $B'\ll \Lambda B^2$ and then $\Lambda \ll \max\{B^{-2},\sqrt{D\delta/B}\}$ thus results in a bound
\[
\ll 1 + D^{1/2} B^{2-\tau/2+\epsi} 
+ D^{1/2} M^{3/5} B^{2-7\tau/10+\epsi}
\]
for the contribution to $N_{I,1}$ from all major sides. Taking minor sides into account, we get
\[
N_{I,1} \ll
M + D^{1/2} B^{2-\tau/2+\epsi} + D^{1/2} M^{3/5} B^{2-7\tau/10+\epsi}.
\]
The presence of the first term $M$ shows that no loss occurred through the interpolation step above.
Summing this over the dyadic intervals \eqref{eq:fbis_precise} (assuming $C = B^{O(1)}$), we arrive at the following bound.

\begin{lemma}
\label{lem:duality}
For an interval where \eqref{eq:fbis_big} holds, we have
\[
N_{I,1} \ll B^{9/(4\tau)+\epsi} + C^{1/2} B^{2-\tau/2+\epsi} + C^{1/2} B^{2+27/(20\tau)-7\tau/10+\epsi}.
\]
\end{lemma} 

The corresponding bound for an interval $I'$ of type \eqref{eq:fbis_small}, as provided by Lemma \ref{lem:roth_case}, is 
\begin{equation}
\label{eq:lem_roth_3}
N_{I',1} \ll M + MB^{2+\epsi} C^{-2} + M^{1/2} B^{2-\tau/2+\epsi}
\end{equation}
under the assumption \eqref{inflection_assumptions}. We balance these two bounds by  choosing
\[
C = \min\{M^{2/5} B^{\tau/5}, M^{4/25} B^{7\tau/25}\}.
\]
The third term in \eqref{eq:lem_roth_3} is then rendered negligible. Let us denote the total contribution to $N^*_\gamma(F,B)$ from linear auxiliary forms $A_i$ by $N^{(1)}$. Recalling that the number of intervals $I$ under consideration is at most $O_k(1)$, the above elaborations result in the bound

\begin{equation}
\label{eq:lines_final}
\begin{split}
N^{(1)} &\ll M + B^{2+\epsi}(M^{1/5}B^{-2\tau/5} + M^{17/25} B^{-14\tau/25})\\
&\ll B^{9/(4\tau)+\epsi} + B^{2+9/(20\tau)-2\tau/5+\epsi} + B^{2+153/(100\tau)-14\tau/25+ \epsi}.
\end{split}
\end{equation}
We note that the bound in \eqref{eq:lines_final} is
\begin{align*}
&o(B) && \text{for } \tau \geq 2.89..., \\
&o(B^{1/2}) && \text{for } \tau \geq 4.5,\\
&O(B^{9/(4\tau) + \epsi}) && \text{for } \tau \geq 3.83....
\end{align*}

In the general case, Lemma \ref{lem:roth_case} gives
\begin{equation}
\label{eq:lem_roth_nu}
N_{I,1} \ll M + MB^{2+\epsi} C^{-\frac{d+1}{2(d-2)}} + M^{1/2} B^{2-\tau/2+\epsi}.
\end{equation}
Since 
\[
\frac{d+1}{2(d-2)} = \frac{1}{2} + \frac{3}{2(d-2)}
\]
is decreasing in $d$, we may replace $d$ by $\nu_F$ in the above estimate. We simplify the optimisation in this case by concentrating on the middle term in Lemma \ref{lem:duality}, which dominates the third one as soon as $\tau$ is at least $ \frac{3\sqrt{3}}{2} \approx 2.6$. Thus we put
\[
C = (MB^{\tau/2})^{\frac{2(\nu-2)}{2\nu-1}}.
\]
This gives
\begin{equation}
\label{eq:lines_final_nu}
\begin{split}
N^{(1)} &\ll M^{1+\epsi} + M^{\frac{\nu-2}{2\nu-1}}B^{2+\epsi-\frac{\tau}{2} + \frac{\tau(\nu-2)}{2(2\nu-1)}} \\
&\hphantom{\ll} + M^{\frac{3}{5} + \frac{\nu-2}{2\nu-1}} B^{2+\epsi-\frac{7\tau}{10} + \frac{\tau(\nu-2)}{2(2\nu-1)}}. \end{split}
\end{equation}
(Clearly the third term in \eqref{eq:lem_roth_nu} will give a negligible contribution this time as well.) Noting that the quantity on the right hand side of \eqref{eq:lines_final_nu} is increasing in $\nu$ for the relevant range of $\nu$, we may take the limit as $\nu \to \infty$ to obtain the worst case bound
\begin{equation}
\label{eq:lines_final_worst}
\begin{split}
N^{(1)} &\ll M^{1+\epsi} + M^{1/2}B^{2-\tau/4+\epsi} + M^{11/10} B^{2-9\tau/20+\epsi}.
\end{split}
\end{equation}

We may improve the bound in the general case as follows. Using the fact that
\[
L = \min\{C^{-1/(d-2)},M^{-1}\} \leq (C^{-1/(d-2)})^{(1-\alpha)}(M^{-1})^{\alpha}
\]
for any $\alpha \in [0,1]$, we pick $\alpha = 2/(d+1)$, which allows us to replace the middle term in \eqref{eq:lem_roth_nu} by 
\[
MB^{2+\epsi} L^{(d+1)/2} \leq B^{2+\epsi} C^{-\frac{d-1}{2(d-2)}} \leq B^{2+\epsi} C^{-\frac{\nu-1}{2(\nu-2)}}.
\] 
Putting
\[
C = B^{\frac{\tau(\nu-2)}{2\nu-3}},
\]
we arrive at the bound
\begin{equation}
\label{eq:lines_final_improved}
\begin{split}
N^{(1)} &\ll M^{1+\epsi} + B^{2+\epsi-\frac{\tau}{2} + \frac{\tau(\nu-2)}{2(2\nu-3)}} 
+ M^{\frac{3}{5} + \frac{\nu-2}{2\nu-3}} B^{2+\epsi-\frac{7\tau}{10} + \frac{\tau(\nu-2)}{2(2\nu-3)}} \end{split}.
\end{equation}
Letting $\nu \to \infty$ gives the improved worst case bound
\begin{equation}
\label{eq:lines_final_worst_improved}
\begin{split}
N^{(1)} &\ll M^{1+\epsi} + B^{2-\tau/4+\epsi} + M^{3/5} B^{2-9\tau/20+\epsi}.
\end{split}
\end{equation}
The bound in \eqref{eq:lines_final_worst_improved} is $O(B)$ as soon as $\tau > 4$.

\subsection*{Contribution from conics}

We now treat the contribution to the counting function $N_\gamma(F,B)$ coming from forms $A_i$ from Corollary \ref{cor:determinant} of degree 2 or more. By a standard argument as in \cite[Cor.~1]{Heath-Brown02}, we may assume that $A_i$ is absolutely irreducible.

Suppose that $\deg A_i = 2$, so that $Q(x,y,z) = A_i(x,y,z)$ is an absolutely irreducible, and hence non-singular, quadratic form. Suppose to begin with that $Q$ is not (proportional to) one of the exceptional quadratic forms defining the finitely many osculating conics $O_j$ at sextactic points, introduced in \S \ref{sec:intro}.
We shall use the familiar parameterization of the conic $Q=0$ by quadratic forms, for which we use \cite{Sofos14} as a reference. We may clearly assume that $Q$ has some primitive integral zero $\xxi$ in the box $[-B,B]^3$. There is then, as in \cite[\S6]{Sofos14}, a matrix $\cM\in \mathrm{SL}_3(\ZZ)$ with coefficients of size $O(B)$ such that the quadratic form $Q'(\xx) := Q(\cM\xx)$ has a zero at $(0,1,0)$. Combining this with the parameterization of such quadratic forms expounded in \cite[\S2,3]{Sofos14}, one obtains binary quadratic forms $g_1,g_2,g_3$ with integer coefficients such that each zero $(x,y,z) \in \ZZ^3_{\prim}$  other than $\xxi$ satisfies
\begin{equation}
\label{eq:conic_param}
(x,y,z) = \pm \left(\frac{g_1(s,t)}{\eta},\frac{g_2(s,t)}{\eta},\frac{g_3(s,t)}{\eta}\right)
\end{equation}
for some $(s,t) \in \ZZ^2_{\prim}$, where
\[
\eta = \eta(s,t) = \gcd(g_1(s,t),g_2(s,t),g_3(s,t)).
\]
Furthermore, by \cite[Lemma 2.2]{Sofos14} any natural number $\eta$ occurring as denominator has to be among the divisors of the discriminant of $Q'$. 

\begin{def*}
Given such an $\eta$, let $S_\eta$ be the set of pairs 
$(s,t) \in \ZZ^2_\prim$ satisfying  
\begin{equation*}
\label{eq:eta}
\eta | \gcd(g_1(s,t),g_2(s,t),g_3(s,t)).
\end{equation*}
\end{def*}
One then sees that for each $\eta$ occurring in \eqref{eq:conic_param}, $S_\eta$ is contained in one of the lattices
\[
\Lambda_{\eta,(s_0,t_0)} := \left\{(s,t) \in \ZZ^2 : (s,t) \equiv (\lambda s_0,\lambda t_0) \Mod{\eta} \text{ for some } \lambda \in \ZZ \right\},
\] 
for $(s_0,t_0)$ ranging over at most $4^{\omega(\eta)} = O(B^\epsi)$ pairs in $\ZZ^2_{\prim}$, where $\omega(\cdot)$ denotes the number of distinct prime factors. Let us verify this for $\eta = p^\alpha$, from which the general case follows by the Chinese remainder theorem. Clearly we can assume that at least one of the binary quadratic forms $g_i$ does not vanish identically (mod $p$). There are then two cases to consider. If $g_i(s,t)$ is proportional to the form $st$ (mod $p^\alpha$), then clearly 
$S_{p^\alpha}$ is contained in the union of  $\Lambda_{p^\alpha,(0,1)}$ and $\Lambda_{p^\alpha,(1,0)}$. If $g_i$ does not have this form --- say without loss of generality that the coefficent of $s^2$ is not divisible by $p^\alpha$ --- then the elementary theory of quadratic congruences shows that the congruence $g_i(u,1) \equiv 0 \Mod{p^\alpha}$ can have at most four solutions $u \in \ZZ/p^\alpha\ZZ$. In this case any $(s,t) \in S_{p^\alpha}$ has both coordinates indivisible by $p$, so it follows that $S_{p^\alpha}$ is contained in the union of the corresponding at most four lattices $\Lambda_{p^\alpha,(u,1)}$.

It is also not hard to compute the determinant of $\Lambda_\eta =\Lambda_{\eta,(s_0,t_0)}$ for any pair $(s_0,t_0)$: we have
\begin{equation}
\label{eq:detLambda}
\det(\Lambda_{\eta}) = \eta.
\end{equation}
Indeed, the lattice $\eta \ZZ^2$ is contained in $\Lambda_\eta$ and $[\Lambda_\eta:\eta\ZZ^2] = \eta$, so we have
\[
\det \Lambda_\eta = [\ZZ^2:\Lambda_\eta] = \frac{[\ZZ^2:\eta\ZZ^2]}{[\Lambda_\eta:\eta\ZZ^2]} = \frac{\eta^2}{\eta} = \eta.
\]

It is now time to invoke our results from \S \ref{sec:binary} on binary forms. To this end, consider the binary form 
\[
G(s,t) = F(g_1(s,t),g_2(s,t),g_3(s,t))
\]
of degree $2k$. The discussion above shows that in order for a point $(s,t) \in \ZZ^2_\prim$ to contribute to $N_\gamma(F,B)$, it has to lie in one of $O(B^\epsi)$ lattices $\Lambda_\eta$ of determinant $\eta$, and at the same time satisfy
\[
|G(s,t)| \leq \eta^k B^{\gamma}.
\]
Furthermore, by \cite[Lemma 4.1]{Sofos14} and the fact that the coefficients of $Q$ are polynomially bounded in $B$, we must at least have $|s|,|t| \leq B^\kappa$ for some constant $\kappa$. In the notation of \S \ref{sec:binary}, we then have to estimate the quantity $N_G(\Lambda_\eta,B^\kappa,\eta^kB^\gamma)$. Let us assume for simplicity that $k \geq 5$. Since we have excluded the possibility 
that $Q$ is among the exceptional quadratic forms, we then have
\[
a(G) \leq 5 \leq \frac{2k}{2}
\] 
in the notation of \S \ref{sec:binary}. Proposition \ref{prop:binary} then provides the bound
\[
N_G(\Lambda_\eta,B^\kappa,\eta^kB^\gamma) \ll \left(\frac{\eta B^{\gamma/k}}{\eta} + 1\right) \log B \ll B^{\gamma/k + \epsi} = B^{1-\tau/k + \epsi}. 
\]
Summing the contributions from all lattices $\Lambda$, we see that the total contribution to $N_\gamma(F,B)$ from non-exceptional quadratic auxiliary forms is 
\[
O(B^{9/(4\tau) + 1-\tau/k + \epsi}).
\]

\subsubsection*{Contribution from osculating conics}

Next we treat the contribution from one of the exceptional conics $O_j$. Let $Q(x,y,z) \in \ZZ[x,y,z]$ be an irreducible quadratic form, and let $q(x,y) = Q(x,y,1)$. We may restrict our attention to an interval $I$ as in Lemma \ref{lem:subdivision1}, where $g(x,f(x)) = 0$. By a further subdivision of $I$, we may further restrict to an arc of the conic $q(x,y)=0$ given by $y=w(x)$ for a function $w \in C(I)$. Here, $w$ is infinitely differentiable in the interior of $I$, but may have a  vertical tangent at either of the endpoints. 
Let us define the function
\[
h(x) = w(x) - f(x),
\]
which is algebraic of some degree $2 \leq d \leq 2k$ and $C^\infty$ in the interior of $I$ (\label{com:h}not to be confused with the function $h(y)$ used to define the dual curve before Lemma \ref{lem:huxleyL11}). Let
\[
c := \inf_{x \in I} \max_{1 \leq \ell \leq d} |h^{(\ell)}(x)|.
\]
The constant $c$ is well defined by Lemma \ref{lem:nonvanishing_derivative}, and we have $c \geq c_0 >0$ for some constant $c_0$ depending only on $F$. (Here we are using the fact that $Q$ belongs to a finite list of quadratic forms depending on the form $F$.)

For some fixed $\delta \ll B^{-\tau}$, we can divide the interval $I$ into $O_k(1)$ subintervals where either one of the two inequalities
\begin{equation}
\label{eq:arc_inequalities}
|h(x)| \leq \delta \quad \text{or} \quad  |h(x)| \geq \delta
\end{equation}
holds identically. Letting $J$ be one of those subintervals where the first alternative in \eqref{eq:arc_inequalities} holds, our aim is now to estimate the length of $J$. We make a subdivision of $J$ into at most $8k^4$ intervals $J_{\nu}$ according to Lemma \ref{lem:BP_Lemma6}, where we take
$C_\ell = c/2$ for all $1 \leq \ell \leq d$.

Suppose first that the alternative \eqref{eq:small_deriv} holds on $J_\nu$ for $\ell = 1$. By Lemma \ref{lem:nonvanishing_derivative} and the definition of $c$, the other alternative \eqref{eq:large_deriv} must then hold for some $2 \leq \ell \leq 2k$. Let us take this $\ell$ to be minimal. By Lemma \ref{lem:russian}, it then follows that
\[
|J_\nu| \ll \delta^{1/\ell}.
\]
We may certainly assume that the arc of the conic corresponding to $J_\nu$ contains a rational point $(\alpha,\beta) \in \QQ^2$. Suppose first that $\ell \geq 3$. The bound \eqref{eq:small_deriv} on $h''$ then implies that
\[
|h(x) - (h(\alpha) + h'(\alpha)(x-\alpha))| \ll (x-\alpha)^2. 
\]
Similarly, the bound on $f''$ from Lemma \ref{lem:subdivision1} implies that
\[
|f(x) - (f(\alpha) + f'(\alpha)(x-\alpha)| \ll (x-\alpha)^2.
\]
These two bounds imply that
\[
|w(x) - (\beta+\zeta(x-\alpha))| \ll (x-\alpha)^2 \ll \delta^{2/\ell},
\]
for all $x \in J_\nu$, where $\zeta:=h'(\alpha)$. 
If instead $\ell = 2$, then we have
\[
|x-\alpha| \ll \delta^{1/2} \qquad \text{and} \qquad |h(x)-h(\alpha)| \ll \delta^{1/2},
\]
implying that
\[
|w(x) - \beta| \ll \delta^{1/2}.
\]

Finally, suppose that \eqref{eq:large_deriv} holds for $\ell =1$. Then Lemma \ref{lem:russian} implies that $|J_\nu| \leq 24 \delta/c$, which in turn gives
\begin{align*}
|w(x) - \beta| &\leq |f(x)-f(\alpha)| + |h(x) - h(\alpha)| \\
&\leq |x-\alpha| + 2 \delta \leq \left(24 + \frac{2}{c}\right) \delta.
\end{align*}

To sum up, the rational points that lie on the conic $q(x,y)=0$ and at the same time contribute to $N_\gamma(F,B)$ all lie in one of at most $O_k(1)$ parallelograms, each defined by conditions
\[
|x-\alpha| \leq B_1, \qquad |y-\beta-\zeta(x-\alpha)| \leq B_2,
\]
where either $B_1,B_2 \ll B^{-\tau/2}$ or 
\[
B_1 \ll B^{-\tau/\ell}, \qquad B_2 \ll B^{-2\tau/\ell}
\]
for some $3 \leq \ell \leq 2k$. Switching to the projective viewpoint, we require an estimate for the quantity
\[
\cN\big(Q,(B_1,B_2,B),(\alpha,\beta,\zeta)\big)
\]
say, where for a homogeneous polynomial $H \in \ZZ[x,y,z]$ and constants $B_1,B_2,B_3$, $\alpha_1,\alpha_2,\alpha_3$ we define $\cN(Q,\mathbf{B},\boldsymbol{\alpha})$ to be the number of primitive integer tuples $(x,y,z)$ satisfying
\begin{gather*}
H(x,y,z) = 0,\\
|x-\alpha_1 z| \leq B_1, \quad |y-(\alpha_2-\alpha_3\alpha_1)z - \alpha_3 x| \leq B_2, \quad |z| \leq B_3.
\end{gather*}
To this end we need the following generalisation of \cite[Thm 4.8]{Browning09} from boxes to certain parallelepipeds.
\begin{lemma}
\label{lem:browning_thm_4.8}
If $H$ is a non-singular form of degree $d \geq 2$, then for any $\alpha_i \in \RR$ and any $B_i \geq 1$ we have
\[
\cN(H,\mathbf{B},\boldsymbol{\alpha}) \ll_d (B_1 B_2 B_3)^{1/3}.
\]
\end{lemma}

\begin{proof}
In adapting the proof of \cite[Thm 4.8]{Browning09} to this setting, one needs only to observe that the linear transformation 
\[
(x,y,z) \mapsto (x-\alpha_1 z, y-(\alpha_2-\alpha_3\alpha_1)z - \alpha_3 x,z)
\]
has determinant 1.
\end{proof}

Among the possible bounds for $B_1,B_2$ above, it is clear that the worst case occurs when $B_1 \approx B^{-\tau/(2k)}$ and $B_2 \approx B^{-\tau/k}$. An application of Lemma \ref{lem:browning_thm_4.8} then shows that the contribution to $N_\gamma(F,B)$ from a single quadratic form $Q$, and hence the contribution when $Q$ ranges over the finitely many forms defining the exceptional conics, is at most
\begin{equation}
\label{eq:exceptional_conics}
O(B^{1-\tau/(2k)}).
\end{equation}

\subsection*{Contribution from curves of higher degree}

In the case $\deg A_i \geq 3$, we will simply bound the number of solutions to $A_i= 0$, as it seems difficult to combine the auxiliary equation with the original inequality. Denote by $\cN_i$ the corresponding contribution to $N^*_\gamma(F,B)$. By \cite[Thm. 3] {Heath-Brown02} we have
\[
\cN_i 
\ll B^{2/3 + \epsi}.
\]
To improve this bound, we follow the strategy in \cite[\S 4]{Heath-Brown12} to take advantage of the restrictions
\begin{equation}
\label{eq:parallelepiped}
|x-r_i z| \leq \frac{B}{M},\ |y-s_i z| \leq \frac{B}{M},\ |z| \leq B
\end{equation}
imposed by \eqref{eq:auxiliary+patch}. Thus, define the lattice
\[
\Gamma_i = \left\{\frac{M}{B}(x-r_i z), \frac{M}{B}(y-s_i z), \frac{1}{B} z\right\}
\] 
and let $\gamma_i^{(1)} \leq \gamma_i^{(2)} \leq\gamma_i^{(3)}$ be the successive minima of that lattice. We have
\[
\gamma_i^{(1)}\gamma_i^{(2)}\gamma_i^{(3)} \asymp \det(\Gamma_i) = \frac{M^2}{B^3}
\]
by \cite[Lemma 4.3]{Browning09}. It follows that $\gamma_i^{(1)} \ll M^{2/3}/B$. We also have the obvious lower bound $\gamma_i^{(1)} \geq B^{-1}$.

Now a minimal basis for $\Gamma_i$ as in \cite[Lemma 4.3]{Browning09} gives rise to a unimodular transformation 
\[
(x,y,z) \mapsto (\lambda_1,\lambda_2,\lambda_3)
\]
with the property that $(x,y,z)$ satifies \eqref{eq:parallelepiped} only if
\[
|\lambda_j| \leq L_j \quad \text{for } j=1,2,3,\quad \text{where } L_j \ll \frac{1}{\gamma_i^{(j)}}.
\]
The form $\tilde{A_i}(\lambda_1,\lambda_2,\lambda_3)$ obtained by applying this transformation to $A_i$ is of course irreducible of the same degree, so we may then instead apply \cite[Thm. 3]{Heath-Brown02} with $B$ replaced by $L_1 \ll (\gamma_i^{(1)})^{-1}$ to get
\[
\cN_i \ll (\gamma_i^{(1)})^{-2/3} B^\epsi.
\]
We therefore investigate how often the first successive minimum is of a certain order of magnitude.

\begin{lemma}
\label{lem:successive_min}
For each $B/M^{2/3} \ll L \ll B$, the number of indices $i \leq M$ in Corollary \ref{cor:determinant} such that 
\begin{equation}
\label{eq:L_dyadic}
L \leq (\gamma_i^{(1)})^{-1} \leq 2L
\end{equation}
is at most 
\[
O\left(\frac{B^{2+\epsi}}{L^2}\right).
\]
\end{lemma}

\begin{proof}
Put $v=v_i$, $w= w_i$, $\gamma = \gamma_i^{(1)}$. By  definition of $\gamma$, there exists a vector $(x_1,y_1,z_1)\in \ZZ^3$ such that 
\begin{equation}
\label{eq:gamma}
\max\left\{B^{-1}|M x_1 - vz_1|,B^{-1}|M y_1 - wz_1|,B^{-1}|z_1|\right\} = \gamma.
\end{equation}
By restricting to a square $S$ as in \eqref{eq:partial_lower}, we may assume that once $v$ has been chosen, there are only $O(1)$ admissible choices for $w$. Furthermore, if $z_1 = 0$ in \eqref{eq:gamma}, then $L^{-1} \gg \gamma = M/B$, which contradicts the assumption on $L$, at least if $B$ is large enough. Hence it is enough to count the number of $v,x_1,z_1$ that satisfy
\begin{align}
\label{eq:vz_1}
0 \neq v z_1 &= M x_1 + O\left(\frac{B}{L}\right),\\
z_1 &= O\left(\frac{B}{L}\right).
\end{align}
These conditions imply that $x_1 \ll B/L$, and for each $x_1$ there are $O(B/L)$ choices for the non-zero integer on the right hand side of \eqref{eq:vz_1}. By a divisor function estimate there are then at most $O(B^\epsi)$ choices for $v,z_1$, so in total there are $O(B^{2+\epsi}/L^2)$ possibilities for $v,x_1,z_1$.  
\end{proof}

The contribution to $\sum_{i \leq M} \cN_i$ from all $i$ such that $\deg A_i \geq 3$ and \eqref{eq:L_dyadic} holds is thus
\[
\ll L^{2/3} B^\epsi \min\left\{M,\frac{B^2}{L^2}\right\}. 
\]
In the range $B/M^{2/3} \ll L \ll B$, this quantity is maximised when $L \approx B/M^{1/2}$, so a dyadic summation gives
\[
\sum_{\deg A_i \geq 3} \cN_i 
\ll B^{3/(2\tau) + 2/3 + \epsi}.
\]

These bounds show that the contribution to $N_\gamma(F,B)$ coming from forms $A_i$ of degree at least $2$ is
\[
\ll  B^{9/(4\tau) + 1-\tau/k + \epsi} + B^{3/(2\tau) + 2/3 + \epsi},
\]
the contribution \eqref{eq:exceptional_conics} being negligible.
This expression dominates all the terms in \eqref{eq:lines_final}, as is easily seen. As already observed in \S \ref{sec:intro}, the contribution from the tangent lines $T_i$ can also be subsumed in this bound under the assumption \eqref{inflection_assumptions}. Hence Theorem \ref{thm:main_generic} follows. 

To obtain the more general bound of Theorem \ref{thm:main}, we replace the estimate \eqref{eq:lines_final} with \eqref{eq:lines_final_worst_improved}.

\bibliographystyle{plain}
\bibliography{../ratpoints}

\newcommand{\noop}[1]{}
\begin{thebibliography}{10}

\bibitem{Beresnevich-Dickinson-Velani}
Victor Beresnevich, Detta Dickinson, and Sanju Velani.
\newblock Diophantine approximation on planar curves and the distribution of
  rational points.
\newblock {\em Ann. of Math. (2)}, 166(2):367--426, 2007.
\newblock With an Appendix II by R. C. Vaughan.

\bibitem{Bombieri-Pila}
E.~Bombieri and J.~Pila.
\newblock The number of integral points on arcs and ovals.
\newblock {\em Duke Math. J.}, 59(2):337--357, 1989.

\bibitem{Bombieri_Gubler}
Enrico Bombieri and Walter Gubler.
\newblock {\em Heights in {D}iophantine geometry}, volume~4 of {\em New
  Mathematical Monographs}.
\newblock Cambridge University Press, Cambridge, 2006.

\bibitem{Browning09}
Timothy~D. Browning.
\newblock {\em Quantitative arithmetic of projective varieties.}
\newblock Progress in Mathematics 277. Basel: Birkh\"auser. xi, 160~p., 2009.

\bibitem{Carbery_Christ_Wright}
Anthony Carbery, Michael Christ, and James Wright.
\newblock Multidimensional van der {C}orput and sublevel set estimates.
\newblock {\em J. Amer. Math. Soc.}, 12(4):981--1015, 1999.

\bibitem{Chow16}
Sam Chow.
\newblock A note on rational points near planar curves.
\newblock {\em Acta Arith.}, 177(4):393--396, 2017.

\bibitem{Heath-Brown02}
D.~R. Heath-Brown.
\newblock The density of rational points on curves and surfaces.
\newblock {\em Ann. of Math. (2)}, 155(2):553--595, 2002.

\bibitem{Heath-Brown09}
D.~R. Heath-Brown.
\newblock Sums and differences of three {$k$}th powers.
\newblock {\em J. Number Theory}, 129(6):1579--1594, 2009.

\bibitem{Heath-Brown12}
D.R. Heath-Brown.
\newblock {Square-free values of $n^2+1$.}
\newblock {\em Acta Arith.}, 155(1):1--13, 2012.

\bibitem{Huang15}
Jing-Jing Huang.
\newblock Rational points near planar curves and {D}iophantine approximation.
\newblock {\em Adv. Math.}, 274:490--515, 2015.

\bibitem{Huxley94}
M.~N. Huxley.
\newblock The rational points close to a curve.
\newblock {\em Ann. Scuola Norm. Sup. Pisa Cl. Sci. (4)}, 21(3):357--375, 1994.

\bibitem{Marmon11}
Oscar Marmon.
\newblock Sums and differences of four {$k$}th powers.
\newblock {\em Monatsh. Math.}, 164(1):55--74, 2011.

\bibitem{Schmidt80}
Wolfgang~M. Schmidt.
\newblock {\em Diophantine approximation}, volume 785 of {\em Lecture Notes in
  Mathematics}.
\newblock Springer, Berlin, 1980.

\bibitem{Sofos14}
Efthymios Sofos.
\newblock Uniformly counting rational points on conics.
\newblock {\em Acta Arith.}, 166(1):1--14, 2014.

\bibitem{Thorbergsson_Umehara}
Gudlaugur Thorbergsson and Masaaki Umehara.
\newblock Sextactic points on a simple closed curve.
\newblock {\em Nagoya Math. J.}, 167:55--94, 2002.

\bibitem{Thunder01}
Jeffrey~Lin Thunder.
\newblock Decomposable form inequalities.
\newblock {\em Ann. of Math. (2)}, 153(3):767--804, 2001.

\bibitem{Thunder05}
Jeffrey~Lin Thunder.
\newblock Asymptotic estimates for the number of integer solutions to
  decomposable form inequalities.
\newblock {\em Compos. Math.}, 141(2):271--292, 2005.

\bibitem{Walsh}
Miguel~N. Walsh.
\newblock Bounded rational points on curves.
\newblock {\em Int. Math. Res. Not. IMRN}, 2015(14):5644--5658, 2015.

\end{thebibliography}

\end{document}